\documentclass[a4paper, 11pt, oneside, reqno]{amsart}

\usepackage[T1]{fontenc}
\usepackage[utf8]{inputenc} 
\usepackage{amsmath}
\usepackage{amssymb}
\usepackage[DIV=10]{typearea}
\usepackage{amsthm}
\usepackage{hyperref}
\usepackage{xcolor}
\newtheorem{theorem}{Theorem}[section]
\newtheorem{proposition}[theorem]{Proposition}
\newtheorem{corollary}[theorem]{Corollary}
\newtheorem{lemma}[theorem]{Lemma}
\theoremstyle{definition}
\newtheorem{definition}[theorem]{Definition}
\newtheorem{remark}[theorem]{Remark}

\newcommand{\N}{{\mathbb{N}}}

\newcommand{\T}{{\mathbb{T}}}
\newcommand{\C}{{\mathbb{C}}}

\newcommand{\CA}{{\mathcal{A}}}

\newcommand{\CF}{{\mathcal{F}}}
\newcommand{\CH}{{\mathcal{H}}}
\newcommand{\CS}{{\mathcal{S}}}
\newcommand{\CI}{{\mathcal{I}}}
\newcommand{\CJ}{{\mathcal{J}}}
\newcommand{\CK}{{\mathcal{K}}}
\newcommand{\CL}{{\mathcal{L}}}

\newcommand{\CW}{{\mathcal{W}}}
\newcommand{\CT}{{\mathcal{T}}}

\newcommand{\CB}{{\mathcal{B}}}
\newcommand{\CO}{{\mathcal{O}}}

\newcommand{\CR}{{\mathcal{R}}}
\newcommand{\CX}{{\mathcal{X}}}

\newcommand{\WCB}{{\widehat{\mathcal{B}}}}

\newcommand{\reg}{{\operatorname{reg}}}

\newcommand{\af}{\alpha}
\newcommand{\bt}{\beta}
\newcommand{\gm}{\gamma}

\makeatletter
\def\l@subsection{\@tocline{2}{0pt}{2pc}{5pc}{}}
\makeatother

\title[A Cuntz--Krieger Uniqueness Theorem for $C^*$-algebras of RGBDS]{A Cuntz--Krieger Uniqueness Theorem for $C^*$-algebras of relative generalized Boolean dynamical systems}

\author[T. M. Carlsen]{Toke Meier Carlsen}
\address{Department of Science and Technology, University of the Faroe Islands, Vestara Bryggja 15, FO-100 Tórshavn, The Faroe Islands} \email{toke.carlsen\-@\-gmail.\-com}

\author[E. J. Kang]{Eun Ji Kang}
\address{Research Institute of Mathematics, Seoul National University, Seoul 08826, Korea} \email{kkang3333\-@\-gmail.\-com }

\date{\today}

\begin{document}

\begin{abstract} 
We prove a version of the Cuntz--Krieger Uniqueness Theorem for $C^*$-algebras of arbitrary relative generalized Boolean dynamical systems. We then describe properties of a  $C^*$-algebra of a relative generalized Boolean dynamical system when the underlying Boolean dynamical system satisfies Condition (K). 
We also define a notion of  minimality of a Boolean dynamical system and give sufficient and necessary conditions for the minimality.
 Using these results,  we characterize the generalized Boolean dynamical systems who's $C^*$-algebra  is simple. 
 \end{abstract}

\maketitle

\section{Introduction}
In \cite{CK1980}, Cuntz and Krieger constructed a $C^*$-algebra $\CO_A$ 
generated by $n$ partial isometries 
satisfying certain algebraic conditions arising from
an $n\times n$-matrix $A$ with entries in $\{0,1\}$,  and they proved the uniqueness theorem of $\CO_A$ \cite[Theorem 2.13]{CK1980}. This results says that if the matrix 
 $A$ satisfies a fullness condition (I), then any two families of non-zero partial isometries satisfying the above-mentioned algebraic conditions generate isomorphic $C^*$-algebras. The theorem is now known as the \emph{Cuntz--Krieger uniqueness theorem}. It is fundamental for the theory of Cuntz--Krieger algebras (as the algebras $\CO_A$ are now called) as it was used to prove a simplicity result for Cuntz--Krieger algebras \cite[Theorem 2.14]{CK1980} and a description of the primitive ideal space of $\CO_A$ \cite[Theorem 4.7]{HuRa1997}.

When studying a new class of $C^*$-algebras that contains the class of Cuntz--Krieger algebras, it is therefore one of the main topics to prove a result that extend the above-mentioned Cuntz--Krieger uniqueness theorem to every $C^*$-algebra in the new class. For example, graph algebras, topological graph algebras, higher rank graph algebras, labeled graph $C^*$-algebras and $C^*$-algebras of Boolean dynamical systems are generalizations of Cuntz--Krieger algebras, and generalizations of the Cuntz--Krieger uniqueness theorem have been proven for these classes of algebras (\cite[Corollary 2.12]{DrTo2005}, \cite[Theorem 5.12]{Ka2004}, \cite[Corollary 4.6]{RaSiYe2004}, \cite[Theorem 5.5]{BaPa2009}, \cite[Theorem 9.9]{COP}).

Recalling specifically  the case of $C^*$-algebras of Boolean dynamical systems,  if a Boolean dynamical system $(\CB, \CL,\theta)$ such that $\CB$ and $\CL$ are countable satisfies Condition (L), then any two Cuntz--Krieger representations consisting of nonzero partial isometrics generate isomorphic $C^*$-algebras (\cite[Theorem 9.9]{COP}).

A relative generalized Boolean dynamical system $(\CB, \CL,\theta, \CI_\af; \CJ)$ consists of a Boolean dynamical system $(\CB, \CL,\theta)$ together with a family $(\CI_\af)_{\af \in \CL}$  of ideals in $\CB$  such that $\theta_\af(\CB) \subseteq \CI_\af$ for all $\af \in \CL$ and an ideal $\CJ$ of $\CB_{reg}$.  As the $C^*$-algebra $C^*(\CB, \CL,\theta, \CI_\af; \CJ)$ associated to a relative generalized Boolean dynamical system  is  one of  generalizations of Cuntz-Krieger algebras, this class of $C^*$-algebras contains 
the $C^*$-algebras of generalized Boolean dynamical systems, the  $C^*$-algebras associated with labeled spaces, the $C^*$-algebras of Boolean dynamical systems, the $C^*$-algebras of ultragraphs, the $C^*$-algebras of shift spaces, and graph algebras.
This class of $C^*$-algebras is not a new  as it is shown in \cite{CaK2} that the class of  $C^*$-algebras of relative generalized Boolean dynamical systems is the  same as the class of $C^*$-algebras of  generalized Boolean dynamical systems, but it is very useful to describe the quotient of the $C^*$-algebra of a generalized Boolean dynamical system by a gauge-invariant ideal.

In this paper, we generalize the Cuntz--Krieger uniqueness theorem \cite[Theorem 9.9]{COP} to the $C^*$-algebra of an arbitrary relative  generalized Boolean dynamical system. 
To do that, we first prove a Cuntz--Krieger uniqueness theorem (Theorem~\ref{CKUT for GBDS}) for the $C^*$-algebra of a generalized Boolean dynamical system using a partially defined topological graph  associated to a generalized Boolean dynamical system. Then, we 
 prove a version of the Cuntz--Krieger uniqueness theorem (Theorem~\ref{CKUT for RGBDS}) for the $C^*$-algebra of a relative generalized Boolean dynamical systems using the fact that 
 the class of  $C^*$-algebras of relative generalized Boolean dynamical systems is the  same with the class of $C^*$-algebras of  generalized Boolean dynamical systems. 
We believe that our results are worthwhile in the aspect that we prove a Cuntz--Krieger uniqueness theorem without assumption that $\CB$ and $\CL$ are countable and also  we show that Condition (L) is a necessary condition to apply the Cuntz--Krieger uniqueness theorem.

Secondly, we deal with properties of a $C^*$-algebra of a relative generalized Boolean dynamical system when the underlying Boolean dynamical system satisfies Condition (K).
Generalizing \cite[Theorem 6.3 and Theorem 8.1]{CaK1}, we  prove that 
$(\CB, \CL,\theta)$ satisfies Condition (K) if and only if every ideal of  $C^*(\CB, \CL,\theta, \CI_\af;\CJ)$  is gauge-invariant if and only if $C^*(\CB, \CL,\theta, \CI_\af;\CJ)$  has the (weak) ideal property, and if and only if  $C^*(\CB, \CL,\theta, \CI_\af;\CJ)$ has topological dimension zero.
In \cite[Theorem 6.3 and Theorem 8.1]{CaK1}, this results were proved for a locally finite Boolean dynamical system $(\CB,\CL,\theta)$ with countable $\CB$ and $\CL$. 
The local finiteness condition was needed to use the characterization (\cite[Proposition 10.11]{COP}) of a gauge-invariant ideal  of the $C^*$-algebra of a Boolean dynamical systems  for which each action has compact range and closed domain, and the countability of $\CB$ and $\CL$ was needed to use the Cuntz--Krieger uniqueness theorem for $C^*(\CB,\CL,\theta)$ (\cite[Theorem 9.9]{COP}). 
A characterization of  the gauge-invariant ideals of $C^*(\CB, \CL,\theta, \CI_\af;\CJ)$ is given in \cite[Proposition 7.3]{CaK2}.
 Together with this, the uniqueness theorem proved in subsection \hyperref[CKUT for RGBDS]{3.2} will be  devoted to prove that if  $(\CB, \CL,\theta)$ satisfies Condition (K), then every ideal of  $C^*(\CB, \CL,\theta, \CI_\af;\CJ)$ is gauge-invariant. As a virtue of this,  a generalization of \cite[Theorem 6.3 and Theorem 8.1]{CaK1} will be  given without any conditions.

The third goal of the present paper is to characterize simplicity of the $C^*$-algebras of generalized Boolean dynamical systems. 
In \cite[Theorem 9.16]{COP}, the authors  characterize simplicity for a $C^*$-algebra associated with a Boolean dynamical system $(\CB, \CL,\theta)$ under the assumption that $\CB$ and $\CL$ are countable.
 Also, in \cite[Theorem 3.6]{CW2020}, a characterization of simplicity for a $C^*$-algebra associated with weakly left resolving normal labeled space is given under some countability condition.  In both papers, they realize their $C^*$-algebra as a locally compact Hausdorff \'etale groupoid $C^*$-algebra. The countability condition makes their groupoid $C^*$-algebra second countable. Then they apply the simplicity result of \cite[Theorem 5.1]{BCFA2014}. 
 We in this paper give necessary and sufficient conditions for the simplicity of  $C^*(\CB, \CL,\theta, \CI_\af)$  without any countability conditions, which generalizes both \cite[Theorem 9.16]{COP} and \cite[Theorem 3.6]{CW2020}. 
 The directness of its proof is one of the advantage of our result. Another advantage is that we  give a new characterization of 
   the simplicity of  $C^*(\CB, \CL,\theta, \CI_\af)$  in terms of maximal tails. 
 
 This paper is organized as follows. 
   Section \hyperref[preliminary]{2} contains necessary background on relative generalized Boolean dynamical systems,   partially defined topological graphs and their $C^*$-algebras.   
  In Section \hyperref[CKUT for GBDS]{3.1}, we review the way to define a partially defined  topological graph from a generalized Boolean dynamical system, and define an isomorphism between the $C^*$-algebra of the partially defined  topological graph and the $C^*$-algebra associated to  the generalized Boolean dynamical system (Proposition \ref{isom}). Also, we prove that 
 the Condition (L) of a generalized Boolean dynamical system is equivalent to the topological freeness of the associated partially defined topological graph (Proposition \ref{(L) equiv topologically free}). 
 We then apply these results to prove our Cuntz--Krieger uniqueness theorem. 
  In Section \hyperref[CKUT for RGBDS]{3.2}, we recall that for a relative generalized Boolean dynamical system $(\CB, \CL,\theta, \CI_\af; \CJ)$, there is a generalized Boolean dynamical system $(\CB', \CL,\theta_\af', \CI_\af')$ such  that $C^*(\CB, \CL,\theta, \CI_\af; \CJ)$ and $C^*(\CB', \CL,\theta', \CI_\af')$ are isomorphic, and show that $(\CB, \CL,\theta)$ satisfies Condition (L) if and only if $(\CB', \CL,\theta')$ satisfies Condition (L).  Then we apply the Cuntz--Krieger uniqueness theorem  of $C^*(\CB', \CL,\theta', \CI_\af')$ to have our uniqueness theorem. 
   In Section  \hyperref[Condition (K)]{4}, we state  equivalent conditions for a  $C^*$-algebra of a relative generalized Boolean dynamical system that satisfies Condition (K).
  In Section \hyperref[Minimality and simplicity]{5}, we define a  minimality of a Boolean dynamical system and give a number of equivalent conditions to a Boolean dynamical system being minimal. 
  We then characterize the generalized  Boolean dynamical systems which have simple $C^*$-algebras.

\section{Preliminaries}\label{preliminary} 
We will in this section recall some notation and terminology from \cite{CaK1} and \cite{CaK2}. 
We let $\N_0$ denote the set of nonnegative integers, $\N$ denote the set of positive integers, and let $\T=\{z\in\C:|z|=1\}$. 

\subsection{Boolean algebras}
A {\em Boolean algebra} is a relatively complemented distributive lattice $(\CB,\cap,\cup)$ with least element $\emptyset$. (A Boolean algebra is often called a {\em generalized Boolean algebra}.) If $\CB$ is a Boolean algebra, one can define a binary operation $\setminus: \CB \times \CB \rightarrow \CB$ such that 
$A \cap (B \setminus A)=\emptyset$, $A \cup (B \setminus A)= A \cup B$ for $A, B \in \CB$. Given $A, B \in \CB$, $A \cup B$ is called the {\em union} of $A$ and $B$, $A\cap B$ is called the \emph{intersection} of $A$ and $B$, and $B \setminus A$ is called the \emph{relative complement} of $A$ relative to  $B$. 
A Boolean algebra $\CB$ is called {\em unital} if it has a greatest element $1$, namely there exists $1 \in \CB$ such that $1 \cup A = 1$ and $1 \cap A=A$ for all $A \in \CB$.
(Often, Boolean algebras are assumed to be unital, but, we in this paper do not assume that $\CB$ is unital.)
A partial order $\subseteq$ on $\CB$ is the relation $A \subseteq B \iff  A \cap B =A$ for $A, B \in \CB$.  We say $A$ is  a \emph{subset} of $B$ if $A \subseteq B$. 
Note that $A\cup B$ and $A\cap B$ are the least upper-bound and the greatest lower-bound of $A$ and $B$ with respect to the partial order $\subseteq$. 

A non-empty subset $\CI$ of $\CB$ is called an {\em ideal} if 
 $A \cup B \in \CI$ whenever $A, B \in \CI$, and $\CI$ is lower closed, that is, 
if $A \in \CI$ and $B \subseteq A$, then $B \in \CI$.
For $A \in \CB$, we define $\CI_A:=\{ B \in \CB : B \subseteq A\}$, that is the ideal generated by $A$.

Let $\CI$ be an ideal of $\CB$. For $A, B \in \CB$, we define an equivalent relation by 
\begin{equation*}
A \sim B \iff A \cup A' = B \cup B'~\text{for some}~ A', B' \in \CI.
\end{equation*}
We denote by  $[A]_\CI$ the equivalent class of $A \in \CB$ under $\sim$.
If there is no confusion, we just write $[A]$ instead of $[A]_\CI$. 
The set of all equivalent classes of $\CB$ is denoted by $\CB / \CI$. 
 Then,  $\CB / \CI$ is a Boolean algebra with operations defined by 
\[
[A]\cap [B]=[A\cap B],\ [A]\cup [B]=[A\cup B] ~\text{and}~ [A]\setminus [B]=[A\setminus B].
\]

A non-empty subset $\eta \subseteq \CB$ is called a {\em filter} if 
$\emptyset \notin \eta$, $A \cap B \in \eta$ whenever $A , B \in \eta$ and $\xi$ is upper closed, that is,  if $ A \in \eta$ and $A \subseteq B$, then $B \in \eta$.
A filter is an {\em ultrafilter} if it is a maximal element in the set of filters with respect to inclusion of filter.
For a filter $\xi \subseteq \CB$, $\xi$ is an ultrafilter if and only if it is prime, that is, if  $B,B' \in \CB$ with $B \cup B' \in \eta$, then either $B \in \eta$ or $B' \in \eta$.
 We denote by $\widehat{\CB}$  the set of all ultrafilters of $\CB$. 
  For $A\in\CB$, we let $Z(A):=\{\xi\in\widehat{\CB}:A\in\xi\}$ and we equip $\widehat{\CB}$ with the topology generated by $\{Z(A): A\in\CB\}$. Then $\widehat{\CB}$ is a totally disconnected locally compact Hausdorff space such that each $Z(A)$ is compact and open.


\subsection{Relative generalized Boolean dynamical systems}

 A map $\phi: \CB \rightarrow \CB'$ between two Boolean algebras $\CB$ and $\CB'$ is called a {\em Boolean homomorphism} if $$\phi(A \cap B)=\phi(A) \cap \phi(B), \phi(A \cup B)=\phi(A) \cup \phi(B) ~\text{and}~\phi(A \setminus B)=\phi(A) \setminus \phi(B)$$ for all $A,B \in \CB$. A map $\theta: \CB \rightarrow \CB $ is called an {\em action} on  $\CB$ if it is a Boolean homomorphism with $\theta(\emptyset)=\emptyset$.

Let $\CL$ be a set. We define $\CL^0:=\{\emptyset \}$, $\CL^n:=\{(\beta_1, \dots, \beta_n): \beta_i \in \CL\}$  for $n \in \mathbb{N}$,   and $\CL^*:=\cup_{n \in \N_0} \CL^n$.  
For $\beta=(\beta_1, \dots, \beta_n) \in\CL^n$, we denote $|\bt|:=n$ and  write $\beta_1 \cdots \beta_n$ instead of $(\beta_1, \dots, \beta_n)$.  
Also, for $1\leq i\leq j\leq |\beta|$,  we denote by $\beta_{i,j}$ the sub-word $\beta_i\cdots \beta_j$ of  $\beta$, where $\beta_{i,i}=\beta_i$.
For $\beta=\beta_1 \cdots \beta_n$, $\gamma=\gamma_1 \cdots \gamma_m \in \CL^* \setminus \{\emptyset\}$, we denote by $\bt\gm$ the word $\beta_1 \cdots \beta_n\gamma_1\cdots\gamma_m$. If $\beta=\emptyset$, then $\beta\gamma:=\gamma$, and if $\gamma=\emptyset$, then $\beta\gamma:=\beta$. 
For $k\in\N$, we let $\beta^k:=\beta\beta\cdots\beta$ where the concatenation on the right has $k$ terms, and let $\beta^0:=\emptyset$. 
By $\CL^\infty$ we mean the set of sequences with entries in $\CL$. If $x=(x_1,x_2,\dots)\in\CL^\infty$ and $n\in\N$, then we  let $x_{1,n}$ denote the word $x_1x_2\cdots x_n\in\CL^n$. We also let $x_{1,0}=\emptyset$.

We say that a triple $(\CB,\CL,\theta)$ 
is a {\em Boolean dynamical system} if $\CB$ is a Boolean algebra, $\CL$ is a set, and $\theta:=(\theta_\alpha)_{\alpha \in \CL}$ is a family of actions on $\CB$. 
If $(\CB,\CL,\theta)$ is a Boolean dynamical system and $\beta=\beta_1 \cdots \beta_n \in \CL^*\setminus\{\emptyset\}$, then we let $\theta_\beta: \CB \rightarrow \CB$ be the action defined by 
$	\theta_\beta:=\theta_{\beta_n} \circ \cdots \circ \theta_{\beta_1}$. 
We also let $\theta_\emptyset:=\text{Id}$. 
For $B \in \CB$, we define $
\Delta_B^{(\CB,\CL,\theta)}:=\{\alpha \in \CL:\theta_\alpha(B) \neq
\emptyset \}.
$
We will often just write $\Delta_B$ instead of
$\Delta_B^{(\CB,\CL,\theta)}$. We say that $A \in \CB$ is {\em regular} if for any $ \emptyset \neq B \in \CI_A$, we have $0 < |\Delta_B| < \infty$. 
We denote by $\CB_{\reg}$ the set of all regular sets. Note that $\emptyset \in \CB_{reg}$ and $\CB_\reg$ is an ideal of $\CB$.

\begin{definition}
A {\em generalized Boolean dynamical system}  (\cite[Dfinition 3.2]{CaK2}) is a quadruple $(\CB,\CL,\theta,\CI_\alpha)$ where $(\CB,\CL,\theta)$ is a Boolean dynamical system and $\{\CI_\alpha\}_{\alpha\in\CL}$ is a family of ideals in $\CB$ such that $\CR_\alpha\subseteq\CI_\alpha$ for each $\alpha\in\CL$, where 
\[
\mathcal{R}_\alpha:=\{A\in\mathcal{B}:A\subseteq\theta_\alpha(B)\text{ for some }B\in\mathcal{B}\}.
\] A {\em  relative generalized Boolean dynamical system} is a pentamerous $(\CB,\CL,\theta,\CI_\alpha;\CJ)$ where $(\CB,\CL,\theta,\CI_\alpha)$ is a generalized Boolean dynamical system and $\CJ$ is an ideal of $\CB_{\reg}$. 
 A {\em relative Boolean dynamical system} is a quadruple $(\CB,\CL,\theta;\CJ)$ where $(\CB,\CL,\theta)$ is a Boolean dynamical system and $\CJ$ is an ideal of $\CB_{\reg}$.
\end{definition}

\subsection{Saturated hereditary ideals and quotient Boolean dynamical systems}
Suppose $(\CB,\CL,\theta)$ is a Boolean dynamical system. An ideal $\CH$ of $\CB$ is \emph{hereditary} if $\theta_\alpha(A)\in\CH$ whenever $A\in\CH$ and $\alpha\in\CL$, and \emph{saturated} if $A\in\CH$ whenever $A\in\CB_\reg$ and $\theta_\alpha(A)\in\CH$ for every $\alpha\in\Delta_A$. If $(\CB,\CL,\theta;\CJ)$ is a relative Boolean dynamical system, then an ideal $\CH$ of $\CB$ is \emph{$\CJ$-saturated} if $A\in\CH$ whenever $A\in\CJ$ and $\theta_\alpha(A)\in\CH$ for every $\alpha\in\Delta_A$.

Suppose that $(\CB,\CL,\theta;\CJ)$ is a relative Boolean dynamical system and $\CH$ is a  hereditary $\CJ$-saturated ideal of $\CB$. 
If  we define $\theta_{\af}([A]_\CH)=[\theta_{\af}(A)]_\CH$ for all $[A]_\CH \in \CB/\CH$ and $ \af \in \CL$, then  $(\CB / \CH, \CL,\theta)$ becomes a Boolean dynamical system.
 We let 
\[
	\CB_\CH:=\bigl\{A\in\CB:[A]_\CH\in (\CB/\CH)_\reg\bigr\}
\]
(notice that there is a mistake in the definition of $\CB_\CH$ given on Page 24 of \cite{CaK2}). Then $\CB_\CH$ is an ideal of $\CB$ and $\CH\cup\CJ\subseteq\CB_\CH$. If $\CS$ is an ideal of $\CB_\CH$ such that $\CH\cup\CJ\subseteq\CS$ and we let $[\CS]:=\{[A]_\CH:A\in\CS\}$, then $(\CB/\CH,\CL,\theta;[\CS])$ is a relative Boolean dynamical system. Moreover, if $(\CB,\CL,\theta,\CI_\alpha)$ is a generalized Boolean dynamical system and we for each $\alpha\in\CB$ let $[\CI_\alpha]:=\{[A]_\CH:A\in\CI_\alpha\}$, then $(\CB/\CH,\CL,\theta,[\CI_\alpha])$ is a generalized Boolean dynamical system and $(\CB/\CH,\CL,\theta,[\CI_\alpha];[\CS])$ is a relative generalized Boolean dynamical system.

\subsection{The $C^*$-algebra of a relative generalized Boolean dynamical system}
Let $(\CB,\CL,\theta, \CI_\alpha; \CJ)$ be a relative generalized Boolean dynamical system. A {\em  $(\CB, \CL, \theta, \CI_\alpha;\CJ)$-representation} (\cite[Definition 3.3]{CaK2}) consists of a family of projections $\{P_A:A\in\mathcal{B}\}$ and a family of partial isometries $\{S_{\alpha,B}:\alpha\in\mathcal{L},\ B\in\mathcal{I}_\alpha\}$ in a $C^*$-algebra such that for $A,A'\in\mathcal{B}$, $\alpha,\alpha'\in\mathcal{L}$, $B\in\mathcal{I}_\alpha$ and $B'\in\mathcal{I}_{\alpha'}$,
\begin{enumerate}
\item[(i)] $P_\emptyset=0$, $P_{A\cap A'}=P_AP_{A'}$, and $P_{A\cup A'}=P_A+P_{A'}-P_{A\cap A'}$;
\item[(ii)] $P_AS_{\alpha,B}=S_{\alpha,  B}P_{\theta_\alpha(A)}$;
\item[(iii)] $S_{\alpha,B}^*S_{\alpha',B'}=\delta_{\alpha,\alpha'}P_{B\cap B'}$;
\item[(iv)] $P_A=\sum_{\alpha \in\Delta_A}S_{\alpha,\theta_\alpha(A)}S_{\alpha,\theta_\alpha(A)}^*$ for all  $A\in \mathcal{J}$. 
\end{enumerate}
The {\it  $C^*$-algebra of $(\CB,\CL,\theta,\CI_\alpha;\CJ)$}, which we denote by $C^*(\mathcal{B},\mathcal{L},\theta, \CI_\af;\mathcal{J})$, is defined to be the $C^*$-algebra generated by a universal $(\CB, \CL, \theta, \CI_\af;\CJ)$-representation. 
 
A $(\CB, \CL, \theta,\CI_\af;\CB_{reg})$-representation is called a {\it 
$(\CB, \CL, \theta,\CI_\af)$-representation}. 
We write $C^*(\CB, \CL, \theta, \CI_\af)$ for $C^*(\CB, \CL, \theta,\CI_\af;
\CB_{reg})$ and call it the {\it  $C^*$-algebra of
$(\CB,\CL,\theta,\CI_\alpha)$}.


Let $(\CB,\CL,\theta, \CI_\alpha; \CJ)$ be a relative generalized Boolean dynamical system. 
 By the universal property of $C^*(\CB,\CL,\theta,\CI_\alpha;\CJ)=C^*(p_A, s_{\alpha,B})$, there is a strongly continuous action $\gamma:\mathbb T\to {\rm Aut}(C^*(\CB,\CL,\theta, \CI_\alpha;\CJ))$, which we call the {\it gauge action}, such that
\[
\gamma_z(p_A)=p_A   \ \text{ and } \ \gamma_z(s_{\alpha,B})=zs_{\alpha,B}
\]
for $A\in \CB$, $\alpha \in \CL$ and $B \in \CI_\alpha$. We say that an ideal $I$ of $C^*(\CB, \CL, \theta, \CI_\alpha;\mathcal{J})$ is \emph{gauge-invariant} if $\gamma_z(I)=I$ for every $z\in\T$.

For $\af=\af_1\af_2 \cdots \af_n \in \CL^* \setminus \{\emptyset\}$, we define
\[
\CI_\af:=\{A \in \CB : A \subseteq \theta_{\af_2 \cdots \af_n}(B)~\text{for some}~ B \in \CI_{\af_1}\}.
\]
For $\beta =\emptyset$, we let $\CI_\emptyset := \CB$. If $\{P_A,\ S_{\alpha,B}: A\in\CB,\ \alpha\in\CL,\ B\in\CI_\alpha\}$ be a $(\CB,\CL,\theta,\CI_\alpha;\CJ)$-representation, we define for
 $\af=\af_1\af_2 \cdots \af_n \in \CL^* \setminus \{\emptyset\}$ and $A \in \CI_{\af}$, \[
S_{\af,A}:=S_{\af_1,B}S_{\af_2, \theta_{\af_2}(B)}S_{\af_3, \theta_{\af_2\af_3}(B)} \cdots S_{\af_n,A},
\] 
where $B \in \CI_{\af_1}$ is such that $A \subseteq \theta_{\af_2 \cdots \af_n}(B)$. For $\af = \emptyset$, we also define $S_{\emptyset, A}:=P_A$.
It then is known that
$
C^*(P_A,S_{\alpha,B})=\overline{{\rm \operatorname{span}}}\{S_{\af,A}S_{\bt,A}^*: \af,\bt
\in \CL^* ~\text{and}~ A \in \CI_\af\cap \CI_\bt \}$ (see \cite[Remark 3.11]{CaK2}).

\subsection{Gauge-invariant ideals}
If $(\CB, \CL, \theta, \CI_\alpha;\mathcal{J})$ is a relative generalized Boolean dynamical system, $\CH$ is a  hereditary $\CJ$-saturated ideal of $\CB$, and $\CS$ is an ideal of $\CB_\CH$ such that $\CH\cup\CJ\subseteq\CS$, then we let $I_{(\CH,\CS)}$ be the ideal of $C^*(\mathcal{B},\mathcal{L},\theta, \CI_\alpha;\mathcal{J})$ generated by 
\[
	\biggl\{p_A-\sum_{\alpha\in\Delta_{[A]_\CH}}s_{\alpha,\theta_\alpha(A)}s_{\alpha,\theta_\alpha(A)}^*:A\in\CS\biggr\}.
\]
If $I$ is an ideal of $C^*(\mathcal{B},\mathcal{L},\theta, \CI_\alpha;\mathcal{J})$, then we let
\[
	\CH_I:=\{A\in\CB:p_A\in I\}
\]
and
\[
	\CS_I:=\biggl\{A\in\CB_{\CH_I}:p_A-\sum_{\alpha\in\Delta_{[A]_{\CH_I}}}s_{\alpha,\theta_\alpha(A)}s_{\alpha,\theta_\alpha(A)}^*\in I\biggr\}.
\]
Then $\CH_I$ is a  hereditary $\CJ$-saturated ideal of $\CB$, $\CS_I$ is an ideal of $\CB_{\CH_I}$ such that $\CH_I\cup\CJ\subseteq\CS_I$, $I_{(\CH_I,\CS_I)}\subseteq I$, and $I_{(\CH_I,\CS_I)}=I$ if and only if $I$ is gauge-invariant. Moreover, the map $(\CH,\CS)\mapsto I_{(\CH,\CS)}$ is a lattice isomorphism between the lattice of pairs $(\CH,\CS)$ where $\CH$ is a  hereditary $\CJ$-saturated ideal of $\CB$ and $\CS$ is an ideal of $\CB_\CH$ such that $\CH\cup\CJ\subseteq\CS$, with order given by $(\CH_1,\CS_1)\subseteq (\CH_2,\CS_2)\iff \CH_1\subseteq \CH_2\text{ and }\CS_1\subseteq\CS_2$, and the lattice of gauge-invariant ideals of $C^*(\mathcal{B},\mathcal{L},\theta, \CI_\alpha;\mathcal{J})$, and there is for each pair $(\CH,\CS)$ an isomorphism $\phi:C^*(\CB/\CH,\CL,\theta,[\CI_\alpha];[\CS])\to C^*(\mathcal{B},\mathcal{L},\theta, \CI_\alpha;\mathcal{J})/I_{(\CH,\CS)}$ such that $\phi(p_{[A]})=p_A+I_{(\CH,\CS)}$ for $A\in\CB$, and $\phi(s_{\alpha,[B]})=s_{\alpha,B}+I_{(\CH,\CS)}$ for $\alpha\in\CL$ and $B\in\CI_\alpha$.

\subsection{Condition (L)}

Let $(\CB, \CL,\theta)$ be a Boolean dynamical system and let  $\bt=\bt_1 \cdots \bt_{n} \in \CL^*\setminus\{\emptyset
\}$ and $A \in \CB\setminus\{\emptyset\}$.
\begin{enumerate}
\item A pair $(\beta,A)$  is called a \emph{cycle} (\cite[Definition 9.5]{COP}) if $B=\theta_\beta(B)$ for $B\in\CI_A$. 
\item A cycle $(\bt,A)$ has an \emph{exit} (\cite{CaK1}) if  there is a $t \leq n$ and a $B\in\CB$ such that $\emptyset \neq B \subseteq \theta_{\bt_{1,t}}(A)$ and $\Delta_B\ne\{\bt_{t+1}\}$ (where $\bt_{n+1}:=\bt_1$).
\item A cycle $(\beta,A)$  has \emph{no exits} (\cite[Definition 9.5]{COP}) if for $t\in\{1,2,\dots,n\}$ and $ \emptyset \neq B\in\CI_{\theta_{\beta_{1,t}}(A)}$, we have $B \in \CB_{reg}$ with $\Delta_{B}=\{\beta_{t+1}\}$ (where  $\beta_{n+1}:=\beta_1$). 
\item   $(\CB,\CL,\theta)$ is said to satisfy \emph{Condition (L)} (\cite[Definition 9.5]{COP}) if it has no cycle with no exits.
\end{enumerate}

The following lemma will be used to prove Proposition \ref{(L) equiv topologically free}.

\begin{lemma} Let $(\CB,\CL,\theta)$ be a  Boolean dynamical system. If $(\bt,A)$ is a cycle with no exits, where $\bt=\bt_1 \cdots \bt_n \in \CL^* \setminus \{\emptyset\}$ and $A \in \CB \setminus \{\emptyset\}$, then $(\bt_{k+1,n}\bt_{1,k}, \theta_{\bt_{1,k}}(A))$ is a cycle for any $k \in \{1, \cdots, n\}$.
 \end{lemma}
 
 \begin{proof} Let $k \in \{1, \cdots, n\}$. We prove that $B=\theta_{\bt_{k+1,n}\bt_{1,k}}(B)$ for all $B \subseteq \theta_{\bt_{1,k}}(A)$. Take $B \subseteq \theta_{\bt_{1,k}}(A)$. Since $B \subseteq \theta_{\bt_{1,k}}(A)$, we have $\theta_{\bt_{k+1,n}\bt_{1,k}}(B) \subseteq \theta_{\bt_{k+1,n}\bt_{1,k}}(\theta_{\bt_{1,k}}(A))$. Here, $\theta_{\bt_{k+1,n}\bt_{1,k}}(\theta_{\bt_{1,k}}(A))=\theta_{\bt_{1,k}\bt_{k+1,n}\bt_{1,k} }(A)=\theta_{\bt_{1,k}}(\theta_\bt(A))=\theta_{\bt_{1,k}}(A)$. So, we have $\theta_{\bt_{k+1,n}\bt_{1,k}}(B) \subseteq \theta_{\bt_{1,k}}(A)$. On the other hand, since $(\bt,A)$ is a cycle and $\theta_{\bt_{k+1,n}}(B) \subseteq A$, we have $$\theta_\bt(\theta_{\bt_{k+1,n}}(B))=\theta_{\bt_{k+1,n}}(B).$$ Here, $\theta_\bt(\theta_{\bt_{k+1,n}}(B))=\theta_{\bt_{k+1,n}\bt}(B)=\theta_{\bt_{k+1,n}\bt_{1,k}\bt_{k+1,n}}(B)=\theta_{\bt_{k+1,n}}(\theta_{\bt_{k+1,n}\bt_{1,k}}(B))$. So, 
 \begin{align}\label{same image}\theta_{\bt_{k+1,n}}(\theta_{\bt_{k+1,n}\bt_{1,k}}(B))=\theta_{\bt_{k+1,n}}(B).\end{align}
 
 If  $ B \setminus  \theta_{\bt_{k+1,n}\bt_{1,k}}(B) \neq \emptyset $, then $B \setminus \theta_{\bt_{k+1,n}\bt_{1,k}}(B) \in \CB_{reg}$ and $\Delta_{B \setminus \theta_{\bt_{k+1,n}\bt_{1,k}}(B)}=\{\bt_{k+1}\}$  since 
   $(\bt,A)$ is a cycle with no exits. 
 So, $\emptyset \neq \theta_{\bt_{k+1}}(B \setminus \theta_{\bt_{k+1,n}\bt_{1,k}}(B)) \subseteq \theta_{\bt_{1,k+1}}(A) $. Then again, since $(\bt,A)$ is a cycle with no exits, $\theta_{\bt_{k+1}}(B \setminus \theta_{\bt_{k+1,n}\bt_{1,k}}(B)) \in \CB_{reg}$ and $
 \Delta_{\theta_{\bt_{k+1}}(B \setminus \theta_{\bt_{k+1,n}\bt_{1,k}}(B))}=\{\bt_{k+2}\}$. Continuing this process, we have $\theta_{\bt_{k+1,n}}(B \setminus \theta_{\bt_{k+1,n}\bt_{1,k}}(B)) \neq \emptyset$. This contradicts to (\ref{same image}). Thus, $B \subseteq \theta_{\bt_{k+1,n}\bt_{1,k}}(B)$. If $\theta_{\bt_{k+1,n}\bt_{1,k}}(B) \setminus B \neq \emptyset$, the same arguments gives $\theta_{\bt_{k+1,n}}(\theta_{\bt_{k+1,n}\bt_{1,k}}(B) \setminus B) \neq \emptyset$, which contradicts to (\ref{same image}). Thus, $B=\theta_{\bt_{k+1,n}\bt_{1,k}}(B)$.
 \end{proof}

\subsection{Maximal tails}

A \emph{maximal tail} (\cite[Definition 4.1]{CaK1}) of a Boolean dynamical system $(\CB,\CL,\theta)$ is a non-empty subset $\CT$ of $\CB$ such that
\begin{enumerate}
	\item[(T1)] $\emptyset\notin\CT$;
	\item[(T2)] if $A\in\CB$ and $\theta_\alpha(A)\in\CT$ for some $\alpha\in\CL$, then $A\in\CT$;
	\item[(T3)] if $A\cup B\in\CT$, then $A\in\CT$ or $B\in\CT$;
	\item[(T4)] if $A\in\CT$, $B\in\CB$ and $A\subseteq B$, then $B\in\CT$;
	\item[(T5)] if $A\in\CT\cap\CB_\reg$, then there is an $\alpha\in\CL$ such that $\theta_\alpha(A)\in\CT$;
	\item[(T6)] if $A,B\in\CT$ then there are $\beta,\gamma\in\CL^*$ such that $\theta_\beta(A)\cap\theta_\gamma(B)\in\CT$.
\end{enumerate}

\begin{remark} A notion of maximal tail was first introduced in \cite[Definition 4.1]{CaK1}. The condition (T6) above is equivalent to (T5) in \cite[Definition 4.1]{CaK1}. 
\end{remark}

\begin{remark}\label{remark:maximal tail} If $\CT$ is a maximal tail, then $\CH_\CT:=\CB \setminus \CT$ is a hereditary $\CJ$-saturated ideal of $\CB$ for any ideal $\CJ$ of $\CB_{reg}$. 
\end{remark}

 An \emph{ultrafilter cycle} (\cite[Definition 3.1]{CaK1}) of a Boolean dynamical system $(\CB,\CL,\theta)$ is a pair $(\beta,\eta)$, where $\beta\in\CL^*\setminus\{\emptyset\}$ and $\eta\in\WCB$, such that $\theta_\beta(A)\in\eta$ for all $A\in\eta$.
A maximal tail is \emph{cyclic} (\cite[Definition 4.6]{CaK1}) if there is an ultrafilter cycle $(\beta,\eta)$ such that 
\[
	\CT=\{B\in\CB:\theta_\gamma(B)\in\eta\text{ for some }\gamma\in\CL^*\}
\] 
and an $A\in\eta$ such that if $\gamma\in\CL^*\setminus{\emptyset}$, $B\in\CI_A$ and $\theta_\gamma(B)\in\eta$, then $B\in\eta$ and $\gamma=\beta^k$ for some $k\in\N$.

 In \cite[Proposition 6.2]{CaK1}, the following result is 
 stated for Boolean dynamical systems that has compact range and closed domain (see \cite[Subsection 2.2]{CaK1}). However, the proof of \cite[Proposition 6.2]{CaK1}  works  without this assumption and once we replace  elements of the form $s_\mu p_{[C]}$ by $s_{\mu,[C]}$ in the proof of \cite[Proposition 6.2]{CaK1}, we can have the following. 
 For further reference, we record these results here and provide a proof of the parts that needed to be modified.

\begin{proposition}\label{cyclic maximal tails}
Let $(\CB, \CL, \theta, \CI_\af)$ be a generalized Boolean dynamical system. Suppose $(\CB, \CL, \theta)$ has a cyclic maximal tail $\CT$. Then $C^*(\CB/(\CB\setminus\CT),\CL,\theta, [\CI_\af])$ contains an ideal that is not gauge-invariant, and there is a $B\in\CT$ such that $p_{[B]}C^*(\CB/(\CB\setminus\CT),\CL,\theta, [\CI_\af])p_{[B]}$ is isomorphic to $M_n(C(\mathbb{T}))$ for some $n\in\mathbb{N}$, where  we let $[\CI_\alpha]:=\{[A]_{\CB\setminus\CT}:A\in\CI_\alpha\}$. 
\end{proposition}

\begin{proof}
Choose a cyclic maximal tail $\mathcal T$ in $(\CB,\CL,\theta)$. Then there is an ultrafilter cycle $(\alpha,\eta)$ such that $\mathcal{T}=\{B\in\mathcal{B}:\theta_\beta(B)\in\eta\text{ for some }\beta\in\mathcal{L}^*\}$ and an $A\in\eta$ such that if $\beta\in\mathcal{L}^*\setminus\{\emptyset\}$, $B\in\mathcal{I}_A$ and $\theta_\beta(B)\in\eta$, then $B\in\eta$ and $\beta=\alpha^k$ for some $k\in\mathbb{N}$.
One then can see that  $\CB\setminus\CT$ is a hereditary saturated ideal of $\CB$ and that a minimal set $[A]$ admits a cycle $\af$  with no exit in $(\CB/(\CB\setminus\CT), \CL,\theta)$. We also have by \cite[Lemma 6.1]{CaK1} that 
 \begin{align}\label{simple cycle}[\theta_{\af_{1,i}}(A)] \cap [\theta_{\af_{1,j}}(A)]= \emptyset ~~\text{for all}~~  1 \leq i < j \leq n.
\end{align}
Put $B:=\cup_{k=1}^n \theta_{\af_{1,k}}(A)$ with $n=|\alpha|$. Then, for $s_{\mu, [C]}s_{\nu, [C]}^* \in C^*(\CB/(\CB\setminus\CT), \CL,\theta, [\CI_\af])$ where $[C] \in [\CI_\mu]\cap [\CI_\nu]$, if
$$p_{[B]}(s_{\mu, [C]}s_{\nu, [C]}^*)p_{[B]} = 
s_{\mu, [\theta_\mu(B)]\cap [C] \cap [\theta_\nu(B)] }s_{\nu, [\theta_\mu(B)]\cap [C] \cap [\theta_\nu(B)]}^* \neq 0,$$ 
then 
$[\theta_\mu(B)] \cap [\theta_\nu(B)]
\neq \emptyset.$ Thus  $[\theta_\mu(B)] \neq \emptyset$ and  $ [\theta_\nu(B)]\neq \emptyset$, and hence we see that the paths $\mu$, $\nu$ are of the form 
$$\mu=\af_{i ,n}\af^l\af_{1,k},\ 
\nu=\af_{j,n}\af^m\af_{1,k'}$$ 
for some $i,j,l,m \geq 0$ and $1 \leq k, k' \leq n$ since $(\af,[A])$ is a cycle with no exit. Then $ \emptyset \neq[\theta_\mu(B)] \cap [\theta_\nu(B)] = [\theta_{\af_{1, i-1}\mu}(A)] \cap [\theta_{\af_{1,j-1}\nu}(A)]=[\theta_{\af_{1,k}}(A)]\cap [\theta_{\af_{1,k'}}(A)]
$. Thus we have $k=k'$. 
It then    follows  that 
\begin{align*} 
&s_{\mu, [\theta_\mu(B)]\cap [C] \cap [\theta_\nu(B)] }s_{\nu, [\theta_\mu(B)]\cap [C] \cap [\theta_\nu(B)]}^*  \\
& =s_{\af_{i,n}\af^l\af_{1,k}, [\theta_{\af_{1,k}}(A) \cap C]}s_{\af_{j,n}\af^m\af_{1,k}, [\theta_{\af_{1,k}}(A) \cap C]}^* \\
& = s_{\af_{i,n}\af^l\af_{1,k},  [\theta_{\af_{1,k}}(A) \cap C]       }(s_{\af_{k+1},[\theta_{\af_{1,k+1}}(A) \cap \theta_{\af_{k+1}}(C)]}s_{\af_{k+1},[\theta_{\af_{1,k+1}}(A) \cap \theta_{\af_{k+1}}(C)]}^* )  s_{\af_{j,n}\af^m\af_{1,k},  [\theta_{\af_{1,k}}(A) \cap C]  }^* \\
& \hskip 0.5pc \vdots \\
&=s_{\af_{i,n}\af^l\af_{1,n},[\theta_{\af_{1,n}}(A) \cap \theta_{\af_{k+1,n}}(C)]}s_{\af_{j,n}\af^m\af_{1,n},[\theta_{\af_{1,n}}(A) \cap \theta_{\af_{k+1,n}}(C)]}^* \\
&=s_{\af_{i,n}\af^{l+1},[A\cap \theta_{\af_{k+1,n}}(C)]}s_{\af_{j,n}\af^{m+1},[A\cap \theta_{\af_{k+1,n}}(C)]}^* \\
&=s_{\af_{i,n}\af^{l+1},[A]}s_{\af_{j,n}\af^{m+1},[A]}^*. 
\end{align*}  
This means that the hereditary subalgebra  
$p_{[B]}C^*(\CB/(\CB\setminus\CT), \CL,\theta,[\CI_\af])p_{[B]}$ is generated by the elements 
$s_{\af_i,[\theta_{\af_{1,i}}(A)]}$ for $1 \leq i \leq n$. 
Then the same arguments used in  \cite[Proposition 6.2]{CaK1} show that 
$$p_{[B]}C^*(\CB/(\CB\setminus\CT), \CL,\theta,[\CI_\af])p_{[B]} \cong C(\mathbb{T}) \otimes M_n.$$
 It then follows that  $C^*(\CB/(\CB\setminus\CT),\CL,\theta,[\CI_\af])$ contains an ideal that is not gauge-invariant. 
\end{proof}

\subsection{Partially defined topological graphs}
For a locally compact space $X$, we denote by $\widetilde{X}$ the one-point compactification of $X$.
\begin{definition}(\cite[Definition 8.2]{Ka2021}) A {\it partially defined} topological graph is a quadruple $E=(E^0,E^1,d,r)$ where $E^0$ and $E^1$ are locally compact spaces, $d:E^1 \to E^0$ is a local homeomorphism, and $r$ is a continuous map from an open subset $\text{dom}(r) $ of $ E^1$ to $E^0$ satisfying that the map $\tilde{r}:E^1 \to \widetilde{E^0}$ defined by 
$$\tilde{r}(e)=\left\{ 
\begin{array}{ll}
    r(e) & \hbox{if\ } e \in \operatorname{dom}(r), \\
    \infty & \hbox{if\ }e \notin \operatorname{dom}(r)
    \end{array}
\right.$$ is continuous. 
\end{definition}

Let $E$ be a partially defined topological graph. We recall the construction of the $C^*$-algebra $\CO(E)$. 
For $p \in C(E^1)$, we define a map $\left<p,p\right>:E^0 \to [0,\infty]$  by $\left<p,p\right>(v):=\sum_{e\in d^{-1}(v)}|p(e)|^2$ for $v \in E^0$. 
 Then, the set $C_d(E^1):=\{p\in C(E^1):\left<p,p\right>\in C_0(E^0)\}$ is a  Hilbert $C_0(E^0)$-module via
\[\left<p,q\right>(v)=\sum_{e\in d^{-1}(v)}\overline{p(e)}q(e),\]
and
\[(pa)(e):=p(e)a(d(e)),\]
where $p,q\in C_d(E^1)$, $a\in C_0(E^0)$, $v\in E^0$ and $e\in E^1$. Define a left action $\pi_r: C_0(E^0) \to \CL(C_d(E^1))$by
$$(\pi_r(a)p)(e)=\left\{ 
\begin{array}{ll}
    a(r(e))p(e) & \hbox{if\ } e \in \operatorname{dom}(r), \\
    0 & \hbox{if\ }e \notin \operatorname{dom}(r)
    \end{array}
\right.$$
for $a\in C_0(E^0)$, $p\in C_d(E^1)$ and $e\in E^1$.
Then, we have a $C^*$-correspondence $C_d(E^1)$ over $C_0(E^0)$.

A  {\it Toeplitz $E$-pair} (cf, \cite[Definition 2.2]{Ka2004}) on a $C^*$-algebra $\CA$ is a pair of maps $T=(T^0, T^1)$, where $T^0: C_0(E^0) \to \CA$ is a $*$-homomorphism and $T^1: C_d(E^1) \to \CA$ is a linear map, satisfying 
\begin{enumerate}
\item $T^1(p)^*T^1(q)=T^0(\left\langle p,q\right\rangle)$ for $p,q \in C_d(E^1) $,
\item $T^0(a)T^1(p)=T^1(\pi_r(a)p)$ for $a \in C_0(E^0) $ and $p \in C_d(E^1) $.
\end{enumerate}
By $C^*(T^0,T^1)$ we mean the $C^*$-subalgebra of $\CA$ generated by the Toeplitz $E$-pair $(T^0,T^1)$. 

For a Toeplitz $E$-pair $(T^0,T^1)$, we define a $*$-homomorphism $\Phi: \CK(C_d(E^1)) \to \CA$ by $\Phi(\Theta_{p,q})=T^1(p)T^1(q)^*$ for $p,q \in C_d(E^1)$, where the operator $\Theta_{p,q} \in  \CK(C_d(E^1))  $ is defined by $\Theta_{p,q}(r)= p \left<q,r\right >$ for $r \in C_d(E^1)$.

We define the following subsets of $E^0$(cf, \cite[Definition 2.6]{Ka2004}):
\begin{align*}E^0_{sce} & :=\{v \in E^0: \exists V ~\text{neighborhood of}~v ~\text{such that}~ r^{-1}(V)=\emptyset \}, \\
E^0_{fin}&:=\{v \in E^0: \exists V ~\text{neighborhood of}~v ~\text{such that}~ r^{-1}(V)~\text{is compact} \},\\
E^0_{rg}&:=E^0_{fin} \setminus \overline{E^0_{sce}},\\
E^0_{sg} & :=E^0 \setminus E^0_{rg}.
\end{align*}
 A Toeplitz $E$-pair $(T^0,T^1)$ is called a {\it Cuntz-Krieger E-pair} (cf, \cite[Definition 2.9]{Ka2004}) if $T^0(f)=\Phi(\pi_r(f))$ for all $f \in C_0(E^0_{rg})$. 

We denote by $\CO(E)$ the $C^*$-algebra generated by the universal Cuntz-Krieger E-pair $(t^0,t^1)$. Note that  $\CO(E)$ is generated by $\{ t^0(a): a \in C_0(E^0)\}$ and $\{t^1(p): p \in C_d(E^1)\}$ and that by the universal property of $\CO(E)$, there exists an action $\bt: \mathbb{T} \curvearrowright \CO(E)$ defined by $\bt_z(t^0(a))=t^0(a)$ and $\bt_z(t^1(p))=zt^1(p)$ for $a \in C_0(E^0)$ and $p \in C_d(E^1)$ and $z \in \mathbb{T}$.

We set $d^0=r^0=id_{E^0}$ and $d^1=d, r^1=r$. For $n \geq 2$, we  define a space $E^n$ of paths with length $n$ by 
\[E^{n}:=\{(e_1,\ldots,e_n)\in \prod_{i=1}^n E^1:d(e_i)=r(e_{i+1})(1\leq i<n)\}\]
which we regard as a subspace of the product space $\prod_{i=1}^n E^1$. 
For convenience, we will usually write $e_1 \cdots e_n$ for $(e_1, \cdots, e_n) \in E^n$.
We  define a domain map $d^n:E^n \to E^0$ by  $d^n(e_1 \cdots e_n)=d^{n-1}(e_n)$,  an open subset $\text{dom}(r^n):=(\text{dom}(r)\times E^1 \times \cdots \times E^1) \cap E^n$ of $E^n$ and a  range map $r^n: \text{dom}(r^n) \to E^0 $ by $r^n(e_1 \cdots e_n)=r^1(e_1)$. It is easy to see that $d^n$ is a local homeomorphism, $r^n$ is a continuous map such that $\widetilde{r^n}:E^n \to \widetilde{E^0}$ defined by
$$\widetilde{r^n}(e_1 \cdots e_n)=\left\{ 
\begin{array}{ll}
    r^n(e_1 \cdots e_n) & \hbox{if\ } e_1 \cdots e_n \in \operatorname{dom}(r^n), \\
    \infty & \hbox{if\ }e_1 \cdots e_n \notin \operatorname{dom}(r^n)
    \end{array}
\right.$$ is continuous. Thus, $(E^0,E^n, d^n, r^n)$ is a partially defined topologial graph. 
Then, we can define a $C^*$-correspondence $C_{d_n}(E^n)$ over $C_0(E^0)$ similarly as $C_d(E^1)$. 
By the same argument used in \cite[Proposition 1.27]{Ka2004}, 
 we have that $C_{d^{n+m}}(E^{n+m}) \cong C_{d^n}(E^n) \otimes C_{d^m}(E^m)$ as $C^*$-correspondence over $C_0(E^0)$ for any $n,m \geq 0$, and that 
$$C_{d^n}(E^n)=\overline{span}\{\xi_1 \otimes \xi_2 \otimes \cdots \otimes \xi_n:\xi_i \in C_d(E^1)\}$$
for $n \geq 1$. 
To ease notations, we write $d,r$ for $d^n,r^n$.

For $n \geq 2$, we define a linear map $T^n: C_d(E^n) \to C^*(T)$ by $$T^n(\xi)=T^1(\xi_1) T^1(\xi_2) \cdots T^1(\xi_n)$$ for $\xi=\xi_1 \otimes \xi_2 \otimes \cdots \otimes \xi_n \in C_d(E^n)$, and a linear map $\Phi^n: \CK(C_d(E^n)) \to C^*(T)$ by $\Phi^n(\Theta_{\xi,\eta})=T^n(\xi)T^n(\eta)^*$, where  $\Theta_{\xi, \eta} \in  \CK(C_d(E^n))$.

\begin{definition}(cf.\cite[Definition 5.3]{Ka2006}) Let $E$ be a partially defined topological graph. A path $e=e_1 \cdots e_n \in E^n$ is called a {\it loop} if $r(e)=d(e)$. The vertex $r(e)=d(e)$ is called the {\it base point} of the loop $e$. A loop $e=e_1\cdots e_n$ is said to be {\it without entrances} if $r^{-1}(r(e_k))=\{e_k\}$ for $k=1, \cdots, n$.
\end{definition}

\begin{definition}(cf.\cite[definition 5.4]{Ka2006}) A partially defined topological graph $E$ is  {\it topologically free} if the set of base points of loops without entrances has an empty interior. 
\end{definition}

Using the date of $d:E^1 \to E^0$, $r: \operatorname{dom}(r) \to E^0$ and the maps $T^n, \Phi^n$ for $n \geq 1$, we can have the following Cuntz--Krieger uniqueness theorem for $C^
*$-correspondences arising from  partially defined topological graphs on the same way 
as topological graphs. We omit its proof. 

\begin{theorem}(cf.\cite[Theorem 6.4]{Ka2006}) For a partially defined topological graph $E$, the following are equivalent:
\begin{enumerate}
\item $E$ is topologically free;
\item the natural surjection $\rho: \CO(E) \to C^*(T)$ is an isomorphism for every injective Cuntz--Krieger $E$-pair $T$;
\item any non-zero ideal $I$ of $\CO(E)$ satisfies $I \cap t^0(C_0(E^0)) \neq 0$.
\end{enumerate}
\end{theorem}

\section{The Cuntz--Krieger uniqueness theorem} \label{CKUT}

We will in this section generalize the Cuntz--Krieger uniqueness theorem \cite[Theorem 9.9]{COP} to the $C^*$-algebra of an arbitrary generalized Boolean dynamical system.

\subsection{A Cuntz--Krieger uniqueness theorem for $C^*(\CB, \CL,\theta, \CI_\af)$}\label{CKUT for GBDS}

We first generalize Cuntz--Krieger uniqueness theorem \cite[Theorem 9.9]{COP} to the $C^*$-algebra of an arbitrary generalized Boolean dynamical system. We consider  a partially defined topological graph $E_{(\CB,\CL,\theta,\CI_\alpha)}$  from an arbitrary  generalized Boolean dynamical system $(\CB,\CL,\theta,\CI_\alpha)$, and show that $C^*(\CB,\CL,\theta,\CI_\alpha)$ and $\CO(E_{(\CB,\CL,\theta,\CI_\alpha)})$ are isomorphic. We then apply the Cuntz--Krieger uniqueness theorem \cite[Theorem 6.14]{Ka2006} of $\CO(E_{(\CB,\CL,\theta,\CI_\alpha)})$.

Let $(\CB, \CL,\theta, \CI_\af)$ be a generalized Boolean dynamical system. We first recall some terminologies to define a partially defined topological graph associated to $(\CB, \CL,\theta, \CI_\af)$.
 Following \cite{CasK1}, we let    $\CW^*=\{\alpha\in\CL^*:\CI_\alpha\neq \{\emptyset\}\}$.   Put $X_\af:=\widehat{\CI_\af}$ for each  $\af \in  \CW^*$ and equip $X_\af$
 with the topology generated by $\{Z(\af, A): A\in\CI_\af\}$, where 
  we let $$Z(\af, A):=\{\CF \in X_\af: A \in \CF\}$$ for $A \in \CI_\af$. 
We also  equip the set $X_\emptyset\cup\{\emptyset\} (=\widehat{\CB}\cup \{\emptyset\} )$  with a suitable topology; if $\CB$ is unital, the topology is such that $\{\emptyset\}$ is an isolated point. If $\CB$ is not unital, then $\emptyset$ plays the role of the point at infinity in the one-point compactification of $X_\emptyset$.

Let $\af, \bt \in \CW^* \setminus \{\emptyset\}$ be such that $\af\bt \in \CW^*$. Define a continuous map
$$f_{\af[\bt]}: X_{\af\bt} \to X_\af ~\text{by}~
 f_{\af[\bt]}(\CF)=\{A \in \CI_\af: \theta_\bt(A) \in \CF\}$$
 for $\CF \in X_{\af\bt}$, and a continuous map $$f_{\emptyset[\bt]}: X_\bt \to X_\emptyset \cup \{\emptyset\} ~\text{by}~f_{\emptyset[\bt]}(\CF)=\{A \in \CB: \theta_\bt(A) \in \CF\}$$
for $\CF \in X_\bt$ (\cite[Lemma 3.23]{CasK1}). 
 
Let $\af, \bt \in \CW^*$ be such that $\af\bt \in \CW^*$. We also define an open  subspace
$$X_{(\af)\bt}:=\{\CF \in X_\bt:\CF \cap \CI_{\af\bt}\neq \emptyset  \}$$ of $X_\bt$ (\cite[Lemma 4.6(vii)]{CasK1}), a continuous map 
 $$g_{(\af)\bt}: X_{(\af)\bt} \to X_{\af\bt} ~\text{by}~  g_{(\af)\bt}(\CF):=     \CF \cap \CI_{\af\bt}  $$
for each  $\CF \in X_{(\af)\bt}$ (\cite[Lemma 4.6(vi)]{CasK1}), and 
 a continuous map 
$$h_{[\af]\bt}:X_{\af\bt} \to X_{(\af)\bt} ~\text{by}~h_{[\af]\bt}(\CF):=\{A\in\CI_\bt:B\subseteq A\text{ for some }B\in\CF\}$$ for  $\CF \in X_{\af\bt}$ (\cite[Lemma 4.8(v)]{CasK1}).
Note that  $X_{(\emptyset)\bt}=X_{\bt}$,   $g_{(\emptyset)\bt}$ and $h_{[\emptyset]\bt}$ are the identity functions on $X_\bt$, and that  $h_{[\af]\bt}: X_{\af\bt} \to X_{(\af)\bt}$ and $g_{(\af)\bt}:X_{(\af)\bt} \to X_{\af\bt}$ are mutually inverses (\cite[Lemma 4.8(iii)]{CasK1}).

We now define a partially defined topological graph from $(\CB, \CL,\theta, \CI_\af)$. Let $$E^0_{(\CB,\CL,\theta,\CI_\alpha)}:=X_\emptyset ~\text{and}~	E^1_{(\CB,\CL,\theta,\CI_\alpha)}:=\bigl\{e^\alpha_\eta:\alpha\in\CL,\ \eta\in X_\alpha\bigr\}
$$
and equip $E^1_{(\CB,\CL,\theta,\CI_\alpha)}$ with the topology generated by 
$\bigcup_{\alpha\in\CL} \{Z^1(\af, B):B\in\CI_\alpha\}, $
where $$Z^1(\af, B):=\{e^\af_\eta: \eta \in X_\af  , B \in \eta\}.$$  Note that $E^1_{(\CB,\CL,\theta,\CI_\alpha)}$ is homeomorphic to the disjoint union of the family $\{X_{\af}\}_{\af\in\CL}$.
 Then, define a local homeomorphism
 $$d:E^1_{(\CB,\CL,\theta,\CI_\alpha)}\to E^0_{(\CB,\CL,\theta,\CI_\alpha)} ~\text{by}~d(e^\alpha_\eta)=h_{[\alpha]\emptyset}(\eta).$$ Put
\begin{align*} \operatorname{dom}(r)&:=\{e^\af_\eta:  \af \in \CL,~ \eta \cap \CR_\af \neq \emptyset\} \subset E^1_{(\CB,\CL,\theta,\CI_\alpha)} \\
&\Big(= \bigcup_{\af \in \CL, A \in \eta \cap \CR_\af} Z^1(\af,A) \Big),
\end{align*}
which is  an open subset of $E^1_{(\CB,\CL,\theta,\CI_\alpha)}$, and define a continuous map $$r:\operatorname{dom}(r)\to E^0_{(\CB,\CL,\theta,\CI_\alpha)} ~\text{by}~r(e^\alpha_\eta)=f_{\emptyset[\af]}(\eta).$$ Then,  the map $\tilde{r}:E^1_{(\CB,\CL,\theta,\CI_\alpha)} \to E^0_{(\CB,\CL,\theta,\CI_\alpha)} \cup \{\emptyset\}$ defined by 
$$\tilde{r}(e)=\left\{ 
\begin{array}{ll}
    r(e) & \hbox{if\ } e \in \operatorname{dom}(r), \\
    \emptyset & \hbox{if\ }e \notin \operatorname{dom}(r)
    \end{array}
\right.$$ is continuous. Thus, $E_{(\CB,\CL,\theta,\CI_\af)}:=(E^0_{(\CB,\CL,\theta,\CI_\alpha)},E^1_{(\CB,\CL,\theta,\CI_\alpha)}, d,r)$ is a partially defined topological graph (see \cite[Proposition 7.1]{CasK1}). To ease notation, we let $E:=E_{(\CB,\CL,\theta,\CI_\af)}$, $E^0:=E^0_{(\CB,\CL,\theta,\CI_\alpha)}$ and $E^1:=E^1_{(\CB,\CL,\theta,\CI_\alpha)}$. 

The following lemmas will be frequently used throughout the paper.  
\begin{lemma}(\cite[Lemma 3.3]{CasK2}) \label{range of path} Let  $\mu=e^{\af_1}_{\eta_1} \cdots e^{\af_n}_{\eta_n} \in E^n$, where $1 \leq n$. Then, we have $$r(\mu)=f_{\emptyset[\af_1 \cdots \af_n]}\big(g_{(\af_1 \cdots \af_{n-1})\alpha_n}(\eta_n) \big).$$ 
\end{lemma}

\begin{lemma}\label{A} Let $\af \in \CL$. For $e^\af_\eta, e^\af_\xi \in X_\af$, we have $d(e^\af_\eta)=d(e^\af_\xi)$ if and only if $\eta=\xi$. 
\end{lemma}

\begin{proof} ($\Leftarrow$) It is claer. 

($\Rightarrow$) 
$h_{[\af]\emptyset}(\eta)=h_{[\af]\emptyset}(\xi)$ implies that $\eta=g_{(\af)\emptyset}(h_{[\af]\emptyset}(\eta)) =g_{(\af)\emptyset}(h_{[\af]\emptyset}(\xi))=\xi $.
\end{proof}

\begin{proposition}\label{isom}
Let $(\CB,\CL,\theta,\CI_\alpha)$ be a generalized Boolean dynamical system and let $E_{(\CB,\CL,\theta,\CI_\af)}=(E^0,E^1, d,r)$ be the associated partially defined topological graph. Then 
\begin{enumerate}
	\item there is an isomorphism $\phi:C^*(\CB,\CL,\theta,\CI_\alpha)\to\CO(E_{(\CB,\CL,\theta,\CI_\af)})$ that maps $p_A$ to $t^0(1_{Z(A)})$ for $A\in\CB$ and $s_{\alpha,B}$ to $t^1(1_{Z^1(\alpha,B)})$ for $\alpha\in\CL$ and $B\in\CI_\alpha$; 
	\item if $\psi$ is a $*$-homomorphism defined on $\CO(E_{(\CB,\CL,\theta,\CI_\alpha)})$, then $\psi\circ t^0$ is injective if and only if $\psi(\phi(p_A))\ne 0$ for all $A\in\CB \setminus \{\emptyset\}$.
\end{enumerate}
\end{proposition}

\begin{proof}(1):   Let $(t^0,t^1)$ be the universal  Cuntz-Krieger $E_{(\CB,\CL,\theta,\CI_\af)}$-pair in a $C^*$-algebra $\CX$.   We claim that $$\{t^0(1_{Z(A)}), t^1(1_{Z^1(\af,B)}): A \in \CB, \af \in \CL ~\text{and}~ B \in \CI_\af\}$$ is a   $(\CB,\CL,\theta, \CI_\af)$-representation  in $\CX$.  Let $A,A'\in\mathcal{B}$, $\alpha,\alpha'\in\mathcal{L}$, $B\in\mathcal{I}_\alpha$ and $B'\in\mathcal{I}_{\alpha'}$. Then, we check the following;
\begin{enumerate}
\item[(i)] It is easy to check that 
$t^0(1_{Z(A)})t^0(1_{Z(A')})=t^0(1_{Z(A \cap A')})$ and 
$t^0(1_{Z(A \cup A')})=t^0(1_{Z(A)})+t^0(1_{Z(A')})-t^0(1_{Z(A \cap A')})$. 
\item[(ii)] 
For $e^\bt_\eta \in E^1$, we compute
\begin{align*}
&(\pi_r(1_{Z(A)})1_{Z^1(\af,B)})(e^\bt_\eta) \\
&= \begin{cases} 1_{Z(A)}(r(e^\bt_\eta))1_{Z^1(\af,B)}(e^\bt_\eta)&\text{if }  e^\bt_\eta \in \operatorname{dom}(r), \\
0&\text{if}~ e^\bt_\eta \notin \operatorname{dom}(r)\\
\end{cases} \\
&=\begin{cases}
1_{Z(A)}(f_{\emptyset[\bt]}(\eta)) &\text{if } e^\bt_\eta \in \operatorname{dom}(r); \ \bt=\af ~\text{and}~ B \in \eta,\\
0&\text{otherwise}\\
\end{cases}\\
&=\begin{cases}
1&\text{if }  e^\bt_\eta \in \operatorname{dom}(r); \ \bt=\af ~\text{and}~ B \in \eta  ~\text{and}~ \theta_\af(A) \in \eta,\\
0&\text{otherwise}.\\
\end{cases}
\end{align*} Since  $B, \theta_\af(A) \in \eta \iff B\cap \theta_\af(A) \in \eta  $, we have $\pi_r(1_{Z(A)})1_{Z^1(\af,B)}=1_{Z^1(\af, B \cap \theta_\af(A))}$. On the other hand, for $e^\bt_\eta \in E^1$,
\begin{align*}  (1_{Z^1(\af,B)}1_{Z(\theta_\af(A))})(e^\bt_\eta)&=1_{Z^1(\af,B)}(e^\bt_\eta)1_{Z(\theta_\af(A))}(d(e^\bt_\eta)) \\
&=\begin{cases}
1_{Z(\theta_\af(A))}(h_{[\bt]\emptyset}(\eta)) &\text{if } \bt=\af ~\text{and}~B \in \eta,\\
0&\text{otherwise}\\
\end{cases} \\
&=\begin{cases}
1 &\text{if } \bt=\af, B \in \eta  ~\text{and}~ \theta_\af(A) \in h_{[\af]\emptyset}(\eta),\\
0&\text{otherwise}\\
\end{cases} \\
&=\begin{cases}
1 &\text{if } \bt=\af ~\text{and}~ B \cap \theta_\af(A) \in \eta,\\
0&\text{otherwise},\\
\end{cases}
\end{align*}
where the last equality follows from the fact that $\theta_\af(A) \in h_{[\af]\emptyset}(\eta) \iff \theta_\af(A) \in \eta $, and that $B, \theta_\af(A) \in \eta \iff B\cap \theta_\af(A) \in \eta$. Thus, we have $1_{Z^1(\af,B)}1_{Z(\theta_\af(A))}=1_{Z^1(\af, B \cap \theta_\af(A))}$.
It then follows that
\begin{align*}t^0(1_{Z(A)})t^1(1_{Z^1(\af,B)})&=t^1(\pi_r(1_{Z(A)})1_{Z^1(\af,B)})=t^1(1_{Z^1(\af, B \cap \theta_\af(A))})\\
&=t^1(1_{Z^1(\af,B)}1_{Z(\theta_\af(A))})=t^1(1_{Z^1(\af,B)})t^0(1_{Z(\theta_\af(A))}).
\end{align*}
\item[(iii)] 
 For $\eta \in E^0$, we first see that \begin{align*} & \left\langle 1_{Z^1(\af,B)},1_{Z^1(\af',B')}\right\rangle (\eta) \\
 &=\sum_{e^\bt_\chi \in E^1;d(e^\bt_\chi)=\eta}1_{Z^1(\af,B)}(e^\bt_\chi)1_{Z^1(\af',B')}(e^\bt_\chi)\\
&=\begin{cases}
 1_{Z^1(\af,B)}(e^\bt_\chi)1_{Z^1(\af',B')}(e^\bt_\chi) &\text{if } \af=\af'=\bt, ~\text{and}~ \chi = \eta \cap \CI_\af,\\
0&\text{otherwise}
\end{cases}\\
&=\begin{cases}
1  &\text{if }  \af=\af'=\bt, \  B, B' \in \chi  ~\text{and}~  \chi = \eta \cap \CI_\af, \\
0&\text{otherwise},
\end{cases}
\end{align*}
where we use Lemma \ref{A} for the second equality. Since $B, B' \in \chi \iff B \cap B' \in \chi \iff B \cap B' \in \eta$, we have 
$\left\langle 1_{Z^1(\af,B)},1_{Z^1(\af',B')}\right\rangle = \delta_{\af,\af'}1_{Z(B \cap B')}$.
Thus, it follows that $$t^1(1_{Z^1(\af,B)})^*t^1(1_{Z^1(\af',B')})=t^0 \big(\left\langle 1_{Z^1(\af,B)},1_{Z^1(\af',B')}\right\rangle \big)=\delta_{\af,\af'}t^0(1_{Z(B \cap B')}).$$
\item[(iv)] Lastly, for the last relation, we first prove that $$\pi_r(1_{Z(A)})=\sum_{\af \in \Delta_A} \Theta_{1_{Z^1(\af,\theta_\af(A))}, 1_{Z^1(\af,\theta_\af(A))}}$$
 for $A \in \CB_{reg}$.  
 For $p \in C_d(E^1)$ and $e \in E^1$, 
 we see that 
\begin{align*}&\Big(\sum_{\af \in \Delta_A} \Theta_{1_{Z^1(\af,\theta_\af(A))}, 1_{Z^1(\af,\theta_\af(A))}}\Big)(p)(e) \\
&=\sum_{\af \in \Delta_A} \Big( 1_{Z^1(\af,\theta_\af(A))} \left\langle 1_{Z^1(\af,\theta_\af(A))}, p \right\rangle \Big) (e)  \\
&=\begin{cases}  1_{Z^1(\af,\theta_\af(A))}(e^\af_\eta)\left\langle 1_{Z^1(\af,\theta_\af(A))}, p \right\rangle(d(e^\af_\eta))  &\text{if } e=e^\af_\eta  ~\text{for}~ \af \in \Delta_A,\\
0&\text{otherwise} 
\end{cases} \\
&= \begin{cases}
\sum_{d(e')=d(e^\af_\eta)}1_{Z^1(\af,\theta_\af(A))}(e')p(e')   &\text{if } e=e^\af_\eta  ~\text{for}~ \af \in \Delta_A ~\text{and}~ \theta_\af(A) \in \eta,\\
0&\text{otherwise} \\
\end{cases}\\
&= \begin{cases}
p(e) &\text{if } e=e^\af_\eta  ~\text{for}~ \af \in \Delta_A ~\text{and}~ \theta_\af(A) \in \eta,\\
0&\text{otherwise},\\
\end{cases}
\end{align*} 
where the last equality follows by Lemma \ref{A}. 
Also,  for $p \in C_d(E^1)$ and $e \in E^1$, we observe that 
\begin{align*} & (\pi_r(1_{Z(A)})p)(e) \\ &=\begin{cases}
1_{Z(A)}(r(e))p(e) &\text{if } e \in \operatorname{dom}(r),\\
0&\text{otherwise}\\
\end{cases} \\
&=\begin{cases}
1_{Z(A)}(f_{\emptyset[\af]}(\eta))p(e^\af_\eta) &\text{if } e \in \operatorname{dom}(r); e=e^\af_\eta  ~\text{for}~ \af \in \Delta_A ,\\
0&\text{otherwise}\\
\end{cases} \\
&= \begin{cases}
p(e) &\text{if } e \in \operatorname{dom}(r); e=e^\af_\eta  ~\text{for}~ \af \in \Delta_A ~\text{and}~ \theta_\af(A) \in \eta,\\
0&\text{otherwise.}\\
\end{cases}
\end{align*}
Thus, we have $\pi_r(1_{Z(A)})=\sum_{\af \in \Delta_A} \Theta_{1_{Z^1(\af,\theta_\af(A))}, 1_{Z^1(\af,\theta_\af(A))}}$ for  $A \in \CB_{reg}$.
 
 Now, let $A \in \CB_{reg}$ and choose $\xi \in Z(A)$. Then, $\xi \in E^0_{rg}$ by \cite[Lemma 7.9]{CasK1}.
 Thus,  we have $1_{Z(A)} \in C_0(E^0_{rg})$. It thus follows that 
\begin{align*} t^0(1_{Z(A)}) &= \Phi(\pi_r(1_{Z(A)})) \\
&= \Phi \Big(\sum_{\af \in \Delta_A} \Theta_{1_{Z^1(\af,\theta_\af(A))}, 1_{Z^1(\af,\theta_\af(A))}} \Big) \\
&=\sum_{\af \in \Delta_A} t^1(1_{Z^1(\af,\theta_\af(A))})t^1(1_{Z^1(\af,\theta_\af(A))})^*.
\end{align*}
\end{enumerate}
Thus, there is a $*$-homomorphism
$$\phi: C^*(\CB,\CL,\theta,\CI_\af) \to C^*(t^0, t^1) $$ given by
$$\phi(p_A)=t^0(1_{Z(A)}) ~\text{and}~ 
\phi(s_{\af,B})=t^1(1_{Z^1(\af,B)})$$
for each $A \in \CB, \af \in \CL$ and $B \in \CI_\af$. Then
 for $A\ne\emptyset$, we have $Z(A)\ne\emptyset$, and hence, $t^0(1_{Z(A)})\ne 0$ for $A\ne\emptyset$ by \cite[Proposition 3.6]{Ka2004}. 
 Hence,  by the gauge-invariant uniqueness theorem  (\cite[Corollary 6.2]{CaK2}), we have $\phi$ is injective. 
 
 Since  $\mathcal{O}_E$ is generated by $\{t^0(a),\ t^1(p):a \in C_0(E^0),\ p \in C_d(E^1)\}$ and  
 $\{1_{Z(A)}:A\in\CB\} $ generates $C_0(E^0)$ and $\{1_{Z^1(\alpha,B)}:\alpha\in\CL,\ B\in\CI_\alpha\}$ generates $C_d(E^1)$, 
  we have  $\phi$ is surjective.  Hence, $C^*(\CB,\CL,\theta, \CI_\af) \cong \CO(E_{(\CB,\CL,\theta, \CI_\af)})$. 

(2): Let $\psi$ be a $*$-homomorphism defined on $\CO(E_{(\CB,\CL,\theta, \CI_\af)})$. Then the results easily follow since $\psi(\phi(p_A))=\psi(t^0(1_{Z(A)}))$ for $A \in \CB$.  
\end{proof}

Let $\xi \in X_{\emptyset}$ be  such that  $\xi \cap \CI_\af \neq \emptyset$ for some $\af=\af_1\af_2 \cdots \af_n \in \CW^*$. Define 
\begin{align*} \xi_n &:=\xi \cap \CI_{\af_n} , \\
\xi_{i}&:=f_{\emptyset[\af_{i+1}]}(\xi_{i+1}) \cap \CI_{\af_{i}}
\end{align*}
for $i=1, \cdots, n-1$.
Then we have  a path $e^{\af_1}_{\xi_1}   \cdots e^{\af_n}_{\xi_n}$
 in $E$ by \cite[Lemma 3.14]{CasK2}. We write such path for  $e(\af,\xi)$. Note then that 
 $$d(e(\af,\xi))=h_{[\af_n]\emptyset}(\xi_n)=h_{[\af_n]\emptyset}(\xi \cap \CI_{\af_n}) =h_{[\af_n]\emptyset} (g_{(\af_n)\emptyset}(\xi))=\xi$$ 
 and 
  $$r(e(\af,\xi))=f_{\emptyset[\af_{1,n}]}(g_{(\af_{1, n-1})\af_n}(\xi_n))=f_{\emptyset[\af]}(\xi \cap \CI_{\af_n} \cap \CI_\af).$$

\begin{lemma}\label{cycle gives a loop} Let $(\CB,\CL,\theta,\CI_\af)$ be a generalized Boolean dynamical system and let $(\bt,A)$ be a cycle, where $\bt=\bt_1 \cdots \bt_n \in \CL^*$.  Then, for each $\xi \in Z(A)$, the path $e(\bt,\xi)=e^{\bt_1}_{\xi_1}   \cdots e^{\bt_n}_{\xi_n}$ is a loop at $\xi$.
\end{lemma}

\begin{proof} We show that $f_{\emptyset[\bt]}(\xi \cap \CI_{\bt_n} \cap \CI_\bt)=\xi$.  Choose $B \in \xi$. Since $A \in \xi$ and $(\bt,A)$ is a cycle, we have $\theta_\bt(A \cap B)=A \cap B \in \xi$, and hence, $\theta_\bt( B) \in  \xi$. 
It is clear that  $\theta_\bt(B) \in \CI_{\bt_n} \cap \CI_{\bt}$. So, $B \in  f_{\emptyset[\bt]}(\xi \cap \CI_{\bt_n} \cap \CI_{\bt})$. Thus, $\xi \subseteq f_{\emptyset[\bt]}(\xi \cap \CI_{\bt_n} \cap \CI_\bt)$. Then the equality follows since both are ultrafilters. 
\end{proof}

\begin{proposition}\label{(L) equiv topologically free} Let $(\CB,\CL,\theta,\CI_\af)$ be a generalized Boolean dynamical system and let $E_{(\CB,\CL,\theta,\CI_\af)}=(E^0,E^1, d,r)$ be the associated partially defined topological graph. Then $(\CB,\CL,\theta)$  satisfies Condition (L)  if and only if $E_{(\CB,\CL,\theta,\CI_\af)}$ is topologically free.
\end{proposition}

\begin{proof} ($\Rightarrow$)  Suppose that $E_{(\CB,\CL,\theta,\CI_\af)}$ is not topologically free.
It then follows from the Baire category theorem that there is a positive integer $n$ and $A \in \CB$ such that $Z(A)$ is nonempty and each $ Z(A)$ is a base point of a simple loop of length $n$ with no entrances.
Let $\eta \in Z(A)$. Then there is a simple loop $\mu:=e^{\bt_1}_{\eta_1} \cdots e^{\bt_n}_{\eta_n}$ such that $r(\mu)=d(\mu)=\eta$. 
Put $\bt:=\bt_1 \cdots \bt_n$.
 We claim that $(\bt, A \cap \theta_\bt(A))$ is a cycle with no exit.
Let $B \subseteq A \cap \theta_\bt(A)$. 
If $B \setminus \theta_\bt(B) \neq \emptyset$, choose $\xi \in \WCB$ such that $B \setminus \theta_\bt(B) \in \xi$. Then, $B, A, \theta_\bt(A) \in \xi$ and $\theta_\bt(B) \notin \xi$. Consider the path $e(\bt,\xi)$. 
Then $d(e(\bt,\xi))=\xi \in Z(A)$. Also, since $\theta_\bt(A) \in \xi \cap \CI_{\bt_n} \cap \CI_\bt$, we have  $A \in r(e(\bt,\xi)) (=f_{\emptyset[\bt]}(\xi \cap \CI_{\bt_n} \cap \CI_\bt))$, and hence, $r(e(\bt,\xi)) \in Z(A)$.
 Since each element in $Z(A)$ is a base point of a simple loop of length $n$ with no entrances, we must have that 
   $d(e(\bt,\xi))=r(e(\bt,\xi))(=\xi)$.  
   Hence,  $B \in r(e(\bt,\xi))$. It means that
   $\theta_\bt(B) \in \xi$, a contradiction. So,  $B \setminus \theta_\bt(B) = \emptyset$. Thus, $B \subseteq \theta_\bt(B)$.
 
 If $\theta_\bt(B) \setminus B \neq \emptyset$, choose $\xi \in \WCB$ such that $\theta_\bt(B) \setminus B  \in \xi$. Then, $ \theta_\bt(B) \in \xi$ and $B \notin \xi$. 
Consider again the path $e(\bt,\xi)$.  
Since $ \theta_\bt(B) \in \xi$, we have  $r(e(\bt,\xi)) \in Z(B) \subseteq Z(A)$.  So, $r(e(\bt,\xi))$ is the base point of a loop of length $n$ with no entrances. 
It means that  $r(e(\bt,\xi))$ is the range of a unique loop of length $n$.
 Since $e(\bt,\xi)$ is a path of length $n$ with range $r(e(\bt,\xi))$ and domain  $\xi$, it follows that $\xi=r(e(\bt,\xi))$, and hence, $B \in \xi$. This is not the case. So, $\theta_\bt(B) \setminus B = \emptyset$. Thus, $B =\theta_\bt(B)$. 
 So, $(\bt, A \cap \theta_\bt(A))$ is a cycle. 
 
 Suppose $k \in \{1,2, \cdots, n\}$, $ \emptyset \neq B \subseteq \theta_{\bt_{1,k}}(A \cap \theta_\bt(A))$ and $\af \in \Delta_B$. Then $\theta_\af(B) \neq \emptyset$, so there is a $\zeta \in \WCB$ such that $\theta_\af(B) \in \zeta$. Since $\theta_\af(B) \subseteq \theta_{\bt_{1,k}\af}(A \cap \theta_\bt(A))$, we have $\theta_{\bt_{1,k}\af}(A \cap \theta_\bt(A)) \in \zeta$. So, $\zeta \cap \CI_{\bt_{1,k}\af} \neq \emptyset$, thus we have the path $e(\bt_{1,k}\af, \zeta)$.  
Then, $r(e(\bt_{1,k}\af, \zeta)) \in Z(A \cap \theta_\bt(A)) \subset Z(A)$. Hence, $\chi:=r(e(\bt_{1,k}\af, \zeta))$ is a base point of a simple loop of length $n$ with no entrances. 
On the other hand, since $(\bt, A \cap \theta_\bt(A))$ is a cycle,  $\chi$ admits a loop $e(\bt, \chi)$  by Lemma \ref{cycle gives a loop}. That means that 
 $$e^{\bt_1}_{\chi_1} \cdots e^{\bt_k}_{\chi_k}e^{\bt_{k+1}}_{\chi_{k+1}}$$
 is the unique path in $E$ of length $k+1$ with range $\chi$. 
  Since $e(\bt_{1,k}\af, \zeta)$ is also a path in $E$ of length $k+1$ with range $\chi$, it follows that $\chi_i=\zeta_i$ for $i=1, \cdots, k+1$ and $\af=\bt_{k+1}$. This shows that the cycle $(\bt, A \cap \theta_\bt(A))$ has no exit. 
 We thus have that $(\CB,\CL,\theta)$ does not satisfy Condition (L).
 
($\Leftarrow$)  Assume that $(\CB,\CL,\theta)$  does not satisfy Condition (L).
There is then a cycle $(\bt,A)$ with no exit, where $\bt=\bt_1 \cdots \bt_n$. 
We claim that each element of $Z(A)$ is the base point of a loop without entrances. 
Suppose $\xi \in Z(A)$. Then by Lemma \ref{cycle gives a loop}(i), 
we have a loop  $e(\bt,\xi)=e^{\bt_1}_{\xi_1}   \cdots e^{\bt_n}_{\xi_n}$ at $\xi$.
If the loop $e(\bt,\xi)$ has an entrance, then there exist $k \in \{1, 2, \cdots, n\}$ and $e^\af_\zeta \in E^1$ $(\af \in \CL, \zeta \in X_\af)$ such that $e^\af_\zeta \neq e^{\bt_k}_{\xi_k}$ and  $r(e^\af_\zeta)=r(e^{\bt_k}_{\xi_k})$.
Here, we claim that if $ \af =\bt_k$, then $ \zeta=\xi_k$.
Since $r(e^{\bt_k}_\zeta)=r(e^{\bt_k}_{\xi_k})$, we have $r(e^{\bt_1}_{\xi_1} \cdots e^{\bt_{k-1}}_{\xi_{k-1}}e^{\bt_k}_{\zeta})=r(e^{\bt_1}_{\xi_1} \cdots e^{\bt_{k-1}}_{\xi_{k-1}}e^{\bt_k}_{\xi_k}),$
which means that 
\begin{align}\label{range}f_{\emptyset[\bt_{1,k}]}(g_{(\bt_{1,k-1})\bt_k}(\zeta))=f_{\emptyset[\bt_{1,k}]}(g_{(\bt_{1, k-1})\bt_k}(\xi_k)).\end{align}
We first show that 
for every $B \subseteq \theta_{\bt_1 \cdots \bt_k}(A)$, if $B \in \xi_k$, then $B \in \zeta$.
If  $B \in \xi_k=f_{\emptyset[\bt_{k+1}]}(\xi_{k+1}) \cap \CI_{\bt_k}$ for $B \subseteq \theta_{\bt_1 \cdots \bt_k}(A)$, then 
$$\theta_{\bt_{k+1}}(B) \in \xi_{k+1}=f_{\emptyset[\bt_{k+2}]}(\xi_{k+2}) \cap \CI_{\bt_{k+1}},$$ 
and then, 
$$\theta_{\bt_{k+1}\bt_{k+2}}(B) \in \xi_{k+2}=f_{\emptyset[\bt_{k+3}]}(\xi_{k+3}) \cap \CI_{\bt_{k+2}}.$$
Continuing this process, one has that 
 $\theta_{\bt_{k+1,n}}(B) \in \xi$. Since $\xi=f_{\emptyset[\bt_{1,k}]}(g_{(\bt_{1, k-1})\bt_k}(\zeta))$ and  $(\bt_{k+1,n}\bt_{1,k}, \theta_{\bt_1 \cdots \bt_k}(A))$ is a cycle,
we have 
$$ B=\theta_{\bt_{k+1,n}\bt_{1,k}}(B) = \theta_{\bt_{1,k}}(\theta_{\bt_{k+1,n}}(B)) \in \zeta.$$
Now, if $\zeta \neq \xi_k$, then there is $B \in \CI_{\bt_k}$ such that $B \in \xi_k$ and $B \notin \zeta$. So, we have $B \cap \theta_{\bt_{1,k}}(A) \in \xi_k$. Since $B \cap \theta_{\bt_{1,k}}(A) \subset \theta_{\bt_{1,k}}(A)$, we have $B \cap \theta_{\bt_{1,k}}(A) \in \zeta$. It then follows that $B \in \zeta$, a contradiction. 
Thus, $\zeta = \xi_k$ if $\af =\bt_k$.

Hence, if $e^\af_\zeta \neq e^{\bt_k}_{\eta_k}$, we have $\af \neq \bt_k$. 
Since  $\theta_{\bt_{1,k}}(A) \in \xi_k$
and
$$f_{\emptyset[\bt_{1,k-1}\af]}(g_{(\bt_{1, k-1})\af}(\zeta))=f_{\emptyset[\bt_{1,k}]}(g_{(\bt_{1, k-1})\bt_k}(\xi_k)),$$
we have  $A \in f_{\emptyset[\bt_{1,k}\af]}(g_{(\bt_{1,k-1})\af}(\zeta))$.  It means that  $\theta_{\bt_{1,k-1}\af}(A) \in \zeta \cap \CI_{\bt_{1,k-1}\af}$. So, 
 $ \theta_{\bt_{1,k-1}\af}(A) \neq \emptyset$. Thus, $\af \in \Delta_{ \theta_{\bt_{1,k-1}}(A)}$. This contradicts to the fact that $(\bt,A)$ is a cycle with no exits. So, the loop $e(\bt,\xi)$ has no entrances.
We thus have that each element of $Z(A)$ is the base point of a loop without entrances, and hence that $E$ is not topologically free. 
 \end{proof}

We are ready to state and prove our Cuntz--Krieger uniqueness theorem for the $C^*$-algebra of a  generalized Boolean dynamical system.

\begin{theorem} \label{thm:CKUT for GBDS}
Let $(\CB,\CL,\theta,\CI_\alpha)$ be a generalized Boolean dynamical system. Then the following are equivalent.
\begin{enumerate}
	\item $(\CB,\CL,\theta)$ satisfies Condition (L).
	\item If $C$ is $C^*$-algebra and $\rho:C^*(\CB,\CL,\theta,\CI_\alpha)\to C$ is a $*$-homomorphism, then $\rho$ is injective if and only if $\rho(p_A)\ne 0$ for each $A\in\CB\setminus\{\emptyset\}$.
		\item If $C$ is $C^*$-algebra and $\rho:C^*(\CB,\CL,\theta,\CI_\alpha)\to C$ is a $*$-homomorphism, then $\rho$ is injective if and only if $\rho(s_{\af,A}s_{\af,A}^*)\ne 0$ for all $\af \in \CL^*$ and all $A\in\CI_\af \setminus\{\emptyset\}$.
	\item Every non-zero ideal of $C^*(\CB,\CL,\theta,\CI_\alpha)$ contains $p_A$ for some $A\in\CB\setminus\{\emptyset\}$.

\end{enumerate}
\end{theorem}

\begin{proof} 
 (1)$\iff$(2): Let $\phi:C^*(\CB,\CL,\theta,\CI_\alpha)\to\CO(E_{(\CB,\CL,\theta,\CI_\alpha)})$ be the isomorphism from Proposition~\ref{isom}. Then $\psi\mapsto\psi\circ\phi$ is a bijection between the class of $*$-homomorphisms defined on $\CO(E_{(\CB,\CL,\theta,\CI_\alpha)})$ and the class of $*$-homomorphisms defined on $C^*(\CB,\CL,\theta,\CI_\alpha)$ sucha tht $\psi\circ t^0$ is injective if and only if $\psi(\phi(p_A))\ne 0$ for all $A\in\CB$. The map $\psi\mapsto (\psi\circ t^0,\psi\circ t^1)$ is a bijection between the class of $*$-homomorphisms defined on $C^*(\CB,\CL,\theta,\CI_\alpha)$ such that $\psi\circ t^0$ is injective and the class of injective Cuntz--Krieger $E_{(\CB,\CL,\theta,\CI_\alpha)}$-pairs. The results therefore follows from Proposition~\ref{(L) equiv topologically free} and \cite[Theorem 6.14]{Ka2006}.

(2)$\implies$(3): The ``only if'' part is clear. To prove the ``if'' part, assume 
$\rho(s_{\af,A}s_{\af,A}^*)\ne 0$ for all $\af \in \CL^*$ and all $\emptyset \neq A\in\CI_\af$. Taking $\af =\emptyset$, we have $\rho(p_A)=\rho(s_{\emptyset,A}s_{\emptyset,A}^*) \neq 0$ for all $\emptyset \neq A\in\CI_\emptyset (=\CB)$.
Thus, by (2), $\rho$ is injective.

(3)$\implies$(2): The ``only if'' part is trivial. To prove the ``if'' part, suppose 
$\rho(p_A)\ne 0$ for each $A\in\CB\setminus\{\emptyset\}$. We show that 
$\rho(s_{\af,A}s_{\af,A}^*)\ne 0$ for all $\af \in \CL^*$ and all $A\in\CI_\af \setminus\{\emptyset\}$.
Assume to the contrary that $\rho(s_{\af,A} s_{\af,A}^*) = 0$ for some $\alpha\in\CL^*$ and some $\emptyset \neq A \in \CI_\af$. Then 
$$\rho(p_A)=\rho(s_{\af,A}^*s_{\af,A}s_{\af,A}^*s_{\af,A})=\rho(s_{\af,A}^*)\rho(s_{\af,A}s_{\af,A}^*)\rho(s_{\af,A})=0,$$
a contradiction. So, it follows by (3) that $\rho$ is injective. 

(2)$\implies$(4): Let $I$ be a non-zero ideal of $C^*(\CB,\CL,\theta,\CI_\alpha)$. Then the quotient map from $C^*(\CB,\CL,\theta,\CI_\alpha)$ to $C^*(\CB,\CL,\theta,\CI_\alpha)/I$ is a non-injective $*$-homomorphism. It therefore follows from (2) that $p_A\in I$ for some $A\in\CB\setminus\{\emptyset\}$. 

(4)$\implies$(2): Let $\rho:C^*(\CB,\CL,\theta,\CI_\alpha)\to C$ be a $*$-homomorphism. It is obvious that if $\rho$ is injective, then $\rho(p_A)\ne 0$ for each $A\in\CB\setminus\{\emptyset\}$. Conversely if $\rho(p_A)\ne 0$ for each $A\in\CB\setminus\{\emptyset\}$, then it follows from (3) that $\ker\rho=\{0\}$ and thus that $\rho$ is injective. 
\end{proof}

As a corollary, we get the following strengthening of \cite[Theorem 9.9]{COP} and \cite[Theorem 2.5]{CaK1}.

\begin{corollary}
Let $(\CB,\CL,\theta)$ be a Boolean dynamical system. Then the following three conditions are equivalent.
\begin{enumerate}
 	\item $(\CB,\CL,\theta)$ satisfies Condition (L).
 	\item A $*$-homomorphism $\pi:C^*(\CB,\CL,\theta)\to B$ is injective if and only if $\pi(p_A)\ne 0$ for all $\emptyset \ne A\in\CB$.
 	\item A $*$-homomorphism $\pi:C^*(\CB,\CL,\theta)\to B$ is injective if and only if $\pi(s_\alpha p_A s_\alpha^*)\ne 0$ for every $\alpha\in\CL^*$ and every $\emptyset \ne A\in\CB$ with $A\subseteq\CR_\alpha$.
 \end{enumerate} 
\end{corollary}

\begin{proof} It follows from Theorem~\ref{thm:CKUT for GBDS} and \cite[Example 4.1]{CaK1}. 
\end{proof}

\subsection{A Cuntz--Krieger uniqueness theorem for $C^*(\CB,\CL,\theta,\CI_\af;\CJ)$}\label{CKUT for RGBDS}

We now prove a Cuntz--Krieger uniqueness theorem for the $C^*$-algebras of relative generalized Boolean dynamical systems. 
Given a relative generalized Boolean dynamical system $(\CB,\CL,\theta,\CI_\af;\CJ)$, it is shown in 
\cite{CaK2} that there is 
 a generalized Boolean dynamical system $(\CB', \CL, \theta', \CI_\af')$ such that $C^*(\CB,\CL,\theta,\CI_\af;\CJ)$ is isomorphic to $C^*(\CB', \CL, \theta', \CI_\af')$.
We recall the construction of $(\CB', \CL, \theta', \CI_\af')$ and the isomorphism between  $C^*(\CB,\CL,\theta,\CI_\af;\CJ)$ and $C^*(\CB', \CL, \theta', \CI_\af')$. Then by  applying the Cuntz--Krieger uniqueness theorem (Theorem \ref{CKUT}) of $C^*(\CB', \CL, \theta', \CI_\af')$, we will have our uniqueness theorem. 

Let $(\CB,\CL,\theta,\CI_\af;\CJ)$ be a  relative generalized Boolean dynamical system   and  let 
\[
\CB'=\{(A,[B]_{\CJ}):A,B\in\CB ~\text{and}~ [A]_{\CB_{reg}}=[B]_{\CB_{reg}}\}.
\] 
Define  \begin{align*} 
(A_1,[B_1]_{\CJ})\cup (A_2,[B_2]_{\CJ})&:=(A_1\cup A_2,[B_1\cup B_2]_{\CJ}),\\
(A_1,[B_1]_{\CJ})\cap (A_2,[B_2]_{\CJ})&:=(A_1\cap A_2,[B_1\cap B_2]_{\CJ}),\\
(A_1,[B_1]_{\CJ})\setminus (A_2,[B_2]_{\CJ})&:=(A_1\setminus A_2,[B_1\setminus B_2]_{\CJ}).
\end{align*}
Then $\CB'$ is a Boolean algebra with the least element $ \emptyset:=(\emptyset, [\emptyset]_\CJ)$.
  For $\alpha\in\CL$, if we define $\theta_\alpha':\CB' \to \CB'$ by 
$$
\theta_\alpha'(A,[B]_{\CJ}):=(\theta_\alpha(A),[\theta_\alpha(A)]_{\CJ}),
$$
then $(\CB', \CL,\theta')$ is a Boolean
dynamical system. Note that $$\CB'_{reg}:=\CB'^{(\CB',
\CL,\theta')}_{reg}=\{(A,\emptyset):A\in\CB_{reg}\}.$$
By \cite[Proposition 6.4]{CaK2}, we see that  
the map $\phi:C^*(\CB,\CL,\theta,\CI_\af;\CJ) \to  C^*(\CB', \CL,\theta', \CI_\af'), $ where $\CI_\af':=\{(A,[A]_{\CJ}):A\in\CI_\alpha\}$ for $\alpha\in\CL$, given by 
\[\rho(p_A)=p_{(A,[A]_\CJ)} ~\text{and}~ \phi(s_{\af,B})=s_{\af, (B, [B]_\CJ)}\]
for all $A \in \CB, \af \in \CL$ and $B \in \CI_\af$ is an isomorphism with the inverse map  $\rho:C^*(\CB', \CL,\theta', \CI_\af') \to C^*(\CB,\CL,\theta,\CI_\af;\CJ) $ given by 
$$\rho(p_{(A,[B]_\CJ)})=p_A+p_C-\sum_{\alpha\in\Delta_C}s_{\alpha,\theta_\alpha(C)}s_{\alpha,\theta_\alpha(C)}^*-p_D+\sum_{\alpha\in\Delta_D}s_{\alpha,\theta_\alpha(D)}s_{\alpha,\theta_\alpha(D)}^*,$$ where $C,D\in\CB_{reg}$ are such that $A\cup C=B\cup D$ and $A \cap C =B \cap D =\emptyset$, and $$\rho(s_{\af,(A,[A]_\CJ)})=s_{\af,A}$$
for all $(A,[B]_\CJ) \in \CB', \af \in \CL$ and $(A, [A]_\CJ) \in \CI_\af'$.

\begin{lemma} \label{R(L) iff G(L)} Let $(\CB,\CL,\theta,\CI_\af;\CJ)$ be a relative generalized Boolean dynamical system. Then, $(\CB, \CL,\theta)$ satisfies Condition (L) if and only if $(\CB', \CL,\theta')$ satisfies Condition (L).
\end{lemma}

\begin{proof}($\Rightarrow$) Assume to the contrary that $(\CB', \CL,\theta')$ does not satisfy Condition (L).
There is then a cycle $(\bt, (A, [B]_\CJ))$ with no exit, where $\bt=\bt_1 \cdots \bt_n$. Since  $(\bt, (A, [B]_\CJ))$ has no exit, it follows that $(A,[B]_\CJ) \in \CB'_{reg}$. So, $A \in \CB_{reg}$ and  $(A, [B]_\CJ)=(A, \emptyset)$.
We claim that $(\bt,A)$ is a cycle with no exit in $(\CB, \CL,\theta)$.
Choose $A' \subseteq A$. Then $(A', \emptyset) \subseteq (A, \emptyset)$, so $(\theta_\bt(A'), \emptyset)=\theta_\bt'(A', \emptyset)=(A', \emptyset)$. Thus, $\theta_\bt(A')=A'$, which means that $(\bt, A)$ is a cycle. 
If $(\bt, A)$ has an exit, there is a $t \leq n$ and a $C\in\CB$ such that $\emptyset \neq C \subseteq \theta_{\bt_{1,t}}(A)$ and $\Delta_C\ne\{\bt_{t+1}\}$ (where $\bt_{n+1}:=\bt_1$). It then easy to see that $\emptyset \neq (C, \emptyset) \subseteq (\theta_{\bt_{1,t}}(A), \emptyset)=\theta'_{\bt_{1,t}}(A, \emptyset)$ and $\Delta_{(C, \emptyset)} \neq \{\bt_{t+1}\}$, which contradicts to the fact that $(\bt, (A, [B]_\CJ))$ has no exit.
Hence, $(\bt,A)$ is a cycle with no exit, a contradiction. Therefore, $(\CB', \CL,\theta')$ satisfies Condition (L).

($\Leftarrow$) Suppose that $(\CB, \CL,\theta)$ does not satisfy Condition (L). Choose a cycle $(\bt,A)$ with no exit, where $\bt=\bt_1 \cdots \bt_n$. Then, $A \in \CB_{reg}$ and $(A, \emptyset) \in \CB_{reg}'$. We claim that $(\bt,(A,\emptyset))$
is a cycle with no exit. 
Let $(A', \emptyset) \subseteq (A, \emptyset)$. Then, $\theta_\bt'(A', \emptyset)=(\theta_\bt(A'), \emptyset)=(A', \emptyset)$. So, $(\bt,(A,\emptyset))$
is a cycle. If $(\bt,(A,\emptyset))$ has an exit, there is 
 a $t \leq n$ and a $(C, \emptyset) \in \CB'$ such that $\emptyset \neq (C, \emptyset) \subseteq (\theta_{\bt_{1,t}}(A), \emptyset)$ and $\Delta_{(C, \emptyset)}\ne\{\bt_{t+1}\}$ (where $\bt_{n+1}:=\bt_1$). Then, $\emptyset \neq C \subseteq \theta_{\bt_{1,t}}(A)$ and $\Delta_C\ne\{\bt_{t+1}\}$, this is not the case since the cycle $(\bt,A)$ has no exit. Thus, the cycle $(\bt,(A,\emptyset))$ has no exit, which is a contradiction. So, $(\CB, \CL,\theta)$ satisfies Condition (L).
\end{proof}

\begin{theorem} \label{thm:CKUT for RGBDS} 
Let $(\CB,\CL,\theta,\CI_\alpha;\CJ)$ be a relative generalized Boolean dynamical system. Then the following are equivalent.
\begin{enumerate}
	\item $(\CB,\CL,\theta)$ satisfies Condition (L).
	\item If $C$ is $C^*$-algebra and $\psi:C^*(\CB,\CL,\theta,\CI_\alpha;\CJ)\to C$ is a $*$-homomorphism, then $\psi$ is injective if and only if the following properties hold:\begin{enumerate}
	\item $\psi(p_A)\neq 0$ for all $\emptyset \neq A\in\CB$, 
	\item $\psi(p_B -\sum_{\af \in \Delta_B}s_{\af, \theta_\af(B)}s_{\af, \theta_\af(B)}^*) \neq 0$ for all $\emptyset \neq B \in \CB_{reg} \setminus \CJ$.
	\end{enumerate}
	\end{enumerate}
\end{theorem}

\begin{proof}
(1)$\implies$(2): The ``only if'' statement is clear. We prove the ``if'' part. Let $\psi:C^*(\CB,\CL,\theta,\CI_\alpha;\CJ)\to C$ be a $*$-homomorphism such that $\psi(p_A)\ne 0$ for all $A\in\CB\setminus\{\emptyset\}$ and 
$\psi(p_B -\sum_{\af \in \Delta_B}s_{\af, \theta_\af(B)}s_{\af, \theta_\af(B)}^*) \neq 0$ for all $B \in \CB_{reg} \setminus \CJ$.
Let $\rho: C^*(\CB', \CL,\theta', \CI_\af') \to C^*(\CB, \CL,\theta,\CI_\af;\CJ) $ be the isomorphism given by $$\rho(p_{(A,[B]_\CJ)})=p_A+p_C-\sum_{\alpha\in\Delta_C}s_{\alpha,\theta_\alpha(C)}s_{\alpha,\theta_\alpha(C)}^*-p_D+\sum_{\alpha\in\Delta_D}s_{\alpha,\theta_\alpha(D)}s_{\alpha,\theta_\alpha(D)}^*,$$ where $C,D\in\CB_{reg}$ are such that $A\cup C=B\cup D$ and $A \cap C =B \cap D =\emptyset$, and $$\rho(s_{\af,(A,[A]_\CJ)})=s_{\af,A}$$
for all $(A,[B]_\CJ) \in \CB', \af \in \CL$ and $(A, [A]_\CJ) \in \CI_\af'$.
Then, $ \psi \circ \rho : C^*(\CB', \CL,\theta', \CI_\af') \to C$ is a $*$-homomorphism such that
$$ \psi \circ \rho (s_{\af,(A,[A]_\CJ)}s_{\af,(A,[A]_\CJ)}^*) 
=\psi(s_{\af,A}s_{\af,A}^*) \neq 0$$
for all $\af \in \CL^*$ and all $\emptyset \neq (A,[A]_\CJ) \in \widetilde{\CI}_\alpha$. In fact, if $ \psi \circ \rho (s_{\af,(A,[A]_\CJ)}s_{\af,(A,[A]_\CJ)}^*) 
=\psi(s_{\af,A}s_{\af,A}^*) =0$ for some  $\af \in \CL^*$ and some $\emptyset \neq (A,[A]_\CJ) \in \widetilde{\CI}_\alpha$, then 
$$\psi(p_A)=\psi(s_{\af,A}^*s_{\af,A}s_{\af,A}^*s_{\af,A})=\psi(s_{\af,A}^*)\psi(s_{\af,A}s_{\af,A}^*)\psi(s_{\af,A})=0$$ for $\emptyset \neq A \in \CI_\af$, a contradiction. 
Since $(\CB, \CL,\theta)$ satisfies Condition (L), $(\CB', \CL,\theta')$ satisfies Condition (L). Thus,  $\psi \circ \rho $ is injective by Theorem \ref{thm:CKUT for GBDS}. Hence, $\psi$ is injective.

 (2)$\implies$(1): Let $\phi:C^*(\CB, \CL,\theta,\CI_\af;\CJ) \to  C^*(\CB', \CL,\theta', \CI_\af')$ be the isomorphism such that $\phi(p_A)=p_{(A,[A]_\CJ)}$ and $\phi(s_{\af,B})=s_{(\af, (B, [B]_\CJ))}$ for $A \in \CB, \af \in \CL$ and $B \in \CI_\af$.
If $C$ is a $C^*$-algebra and $\rho: C^*(\CB', \CL,\theta', \CI_\af') \to C$ be a $*$-homomorphism such that $\rho(p_{(A,[B]_\CJ)}) \neq 0$ for each $\emptyset \neq (A,[B]_\CJ) \in \widetilde{\CB}$, then $\rho \circ \phi: C^*(\CB, \CL,\theta,\CI_\af;\CJ)  \to C$
is a $*$-homomorphism such that 
$$\rho \circ \phi(p_A)=\rho(p_{(A,[A]_\CJ)}) \neq 0$$ for all $\emptyset \neq A \in \CB$, and 
\begin{align*}&\rho \circ \phi(p_B -\sum_{\af \in \Delta_B}s_{\af, \theta_\af(B)}s_{\af, \theta_\af(B)}^*)\\
&=\rho\Big(p_{(B,[B]_\CJ)}-\sum_{\af \in \Delta_B}s_{(\af, (\theta_\af(B), [\theta_\af(B)]_\CJ))}s_{(\af, (\theta_\af(B), [\theta_\af(B)]_\CJ))}^*\Big) \\
&=\rho\Big(p_{(\emptyset,[B]_\CJ)} +p_{(B,\emptyset)} -\sum_{\af \in \Delta_{(B, \emptyset)}}s_{(\af, (\theta_\af(B), [\theta_\af(B)]_\CJ))}s_{(\af, (\theta_\af(B), [\theta_\af(B)]_\CJ))}^*\Big) \\
&=\rho(p_{(\emptyset,[B]_\CJ)}) \\
& \neq 0
\end{align*}
for all $ \emptyset \neq  B \in \CB_{reg} \setminus \CJ $. Thus  
$\rho \circ \phi$ is injective by our assumption. So, $\rho$ is injective, and hence, 
$(\CB', \CL,\theta')$ satisfies Condition (L) by Theorem \ref{thm:CKUT for GBDS}. Therefore,  $(\CB,\CL,\theta)$ satisfies Condition (L) by Lemma \ref{R(L) iff G(L)}.
\end{proof}

\section{Condition (K)}\label{Condition (K)}
Recall from \cite[Definition 5.1]{CaK1} that a Boolean dynamical system $(\CB,\CL,\theta)$ is said to satisfy Condition (K) if there is no pair $( (\beta,\eta), A)$ where $(\beta,\eta)$ is an ultrafilter cycle and $A\in\eta$ such that if $\gamma\in\CL^*\setminus\{\emptyset\}$, $B\in\CI_A$ and $\theta_\gamma(B)\in\eta$, then $B\in\eta$ and $\gamma=\beta^k$ for some $k\in\N$. 
We will now generalize and strengthen the characterization given in \cite[Theorem 6.3 and Theorem 8.1]{CaK1} of when a Boolean dynamical system satisfies Condition (K).  Recall from \cite[Remark 2.1]{PR2007} that a $C^*$-algebra $C$ is said to have {\em the ideal property} if whenever $I$ and $J$ are ideals in $C$ such that $I$ is not contained in $J$, there is a projection in $I \setminus J$; from \cite[Definition 8.1]{PP} that $C^*$-algebra $C$ is said to have {\em the weak ideal property} if whenever $I \subsetneq J$ are ideals in $ \mathcal{K} \otimes C$, where $\mathcal{K}$ denotes the $C^*$-algebra of compact operators on a separable infinite dimensional Hilbert space, then $J/I$ contains a nonzero projection, and from \cite{BP2} that a $C^*$-algebra $C$ is said to have {\em topological dimension zero} if the primitive ideal space of $C$ endowed with the hull-kernel topology has a basis of compact open sets. For $n\in\N$, we let $M_n(C(\T))$ denote the $C^*$-algebra of $n\times n$-matrices of continuous functions from $\T$ to $\C$.

\begin{theorem}\label{equiv:Condition (K)}
Let $(\CB,\CL,\theta,\CI_\alpha;\CJ)$ be a relative generalized Boolean dynamical system. Then the following are equivalent.
\begin{enumerate}
	\item $(\CB,\CL,\theta)$ satisfies Condition (K).
	\item $(\CB,\CL,\theta)$ has no cyclic maximal tails.
	\item If $\CH$ is a hereditary $\CJ$-saturated ideal of $\CB$, then $(\CB/\CH,\CL,\theta)$ satisfies Condition (L).
	\item Every ideal of $C^*(\CB,\CL,\theta,\CI_\alpha;\CJ)$ is gauge-invariant.
	\item $C^*(\CB,\CL,\theta,\CI_\alpha;\CJ)$ has the ideal property.
	\item $C^*(\CB,\CL,\theta,\CI_\alpha;\CJ)$ has the weak ideal property.
	\item The topological dimension of $C^*(\CB,\CL,\theta,\CI_\alpha;\CJ)$ is zero.
	\item $C^*(\CB,\CL,\theta,\CI_\alpha;\CJ)$ has no quotient containing a hereditary $C^*$-subalgebra that is isomorphic to $M_n(C(\T))$ for some $n\in\N$.
\end{enumerate}
\end{theorem}

\begin{proof} (1)$\implies$(2) follows from the definition of a cyclic maximal tail.

$(2)\implies(3)$ follows \cite[Proposition 4.8]{CaK1}. 

$(3)\implies(1)$ follows \cite[Proposition 4.5]{CaK1} and Remark \ref{remark:maximal tail}. 

$(3)\implies (4)$: Suppose $I$ is an ideal of $C^*(\CB,\CL,\theta,\CI_\alpha;\CJ)$. Let 
\[
	\CH_I:=\{A\in\CB:p_A\in I\}
\] 
and 
\[
	\CS_I:=\Bigl\{A\in\CB_{\CH_I}:p_A-\sum_{\alpha\in\Delta{[A]_{\CH_I}}}s_{\alpha,\theta_\alpha(A)}s_{\alpha,\theta_\alpha(A)}^*\in I\Bigr\}
\] 
where $\CB_{\CH_I}:=\bigl\{A\in\CB:[A]_{\CH_I}\in (\CB/\CH_I)_\reg\bigr\}$. Then \cite[Lemma 7.2]{CaK2} says that $\CH_I$ is a hereditary $\CJ$-saturated ideal of $\CB$ and $\CS_I$ is an ideal of $\CB_{\CH_I}$ with $\CH_I\cup\CJ\subseteq\CS_I$. According to \cite[Proposition 7.3]{CaK2}, there is a surjective $*$-homomorphism 
\[
	\phi_I:C^*(\CB/\CH_I,\CL,\theta,[\CI_\alpha];[\CS_I])\to C^*(\CB,\CL,\theta,\CI_\alpha;\CJ)/I
\] 
such that $\phi_I(p_{[A]})=p_A+I$ for $A\in\CB$ and $\phi_I(s_{\alpha,[B]})=s_{\alpha,B}+I$ for $\alpha\in\CL$ and $B\in\CI_\alpha$, where $[\CI_\alpha]=\{[A]:A\in\CI_\alpha\}$ and $[\CS_I]=\{[A]:A\in\CS_I\}$, and $I$ is gauge-invariant if (and only if) $\phi_I$ is injective. Since $\phi_I(p_{[A]})=p_A+I=0$ if and only if $A\in\CH_I$ and 
\[
	\phi_I\Bigl(p_{[A]}-\sum_{\alpha\in\Delta_{[A]_{\CH_I}}}s_{\alpha,\theta_\alpha([A])}s_{\alpha,\theta_\alpha([A])}^*\Bigr)=p_A-\sum_{\alpha\in\Delta_{[A]_{\CH_I}}}s_{\alpha,\theta_\alpha(A)}s_{\alpha,\theta_\alpha(A)}^*+I=0
\] 
if and only if $A\in\CS_I$, it follows from Theorem~\ref{CKUT for RGBDS} that if $(\CB/\CH_I,\CL,\theta)$ satisfies Condition (L), then $\phi_I$ is injective. Thus, $I$ is gauge-invariant.

(4)$\implies$(5): Suppose that every ideal of $C^*(\CB,\CL,\theta,\CI_\alpha;\CJ)$ is gauge-invariant. Let $I$ and $J$ be ideals of $C^*(\CB, \CL,\theta, \CI_\alpha;\CJ)$ such that $I\not\subseteq J$. Since $I$ and $J$ are gauge-invariant, $I=I_{(\CH_I, \CS_I)}$ and $J=J_{(\CH_J,\CS_J)}$ for some hereditary $\CJ$-saturated ideals $\CH_I, \CH_J$ and ideals $\CS_I, \CS_J$ of $\CB $ by \cite[Proposition 7.3]{CaK2}. If $\mathcal{H}_I=\{A\in\CB:p_A\in I\}\not\subseteq \{A\in\CB:p_A\in J\}=\mathcal{H}_J$, then $I \setminus J$ contains a projection. If $\CH_I=\CH_J$, then it follows that 
\begin{align*}
\CS_I&=\Bigl\{A \in \CB_{\CH_I}: p_A -\sum_{\alpha \in \Delta_{[A]}}s_{\alpha,\theta_\alpha(A)}s_{\alpha,\theta_\alpha(A)}^* \in I\Bigr\}	 \\
&\not\subseteq  \Bigl\{A \in \CB_{\CH_I}: p_A -\sum_{\alpha \in \Delta_{[A]}}s_{\alpha,\theta_\alpha(A)}s_{\alpha,\theta_\alpha(A)}^* \in J \Bigr\} =\CS_J.
\end{align*}
Hence, $I \setminus J$ contains a projection. This shows that $C^*(\CB,\CL,\theta,\CI_\alpha;\CJ)$ has the ideal property.

(5)$\implies$(6) follows from \cite[Proposition 8.2]{PP}.

(6)$\implies$(7) follows from \cite[Theorem 2.8]{PP2}.

(7)$\implies$(8): Since the property of having topological dimension zero passes to quotients and hereditary subalgebras, a $C^*$-algebra with topological dimension zero can not have a quotient with a hereditary $C^*$-subalgebra that is isomorphic to $M_n(C(\T))$ for some $n\in\N\setminus\{0\}$.

(8)$\implies$(1): We prove $\lnot (1)\implies \lnot (8)$. Suppose that $(\CB,\CL,\theta)$ does not satisfy Condition (K). Then, by (2) and Proposition \ref{cyclic maximal tails},  there is a cyclic maximal tail $\CT$ in $(\CB,\CL,\theta)$ and a $B\in\CT$ such that $p_{[B]}C^*(\CB/(\CB\setminus\CT),\CL,\theta, [\CI_\af])p_{[B]}$ is isomorphic to $M_n(C(\mathbb{T}))$ for some $n\in\N$. 
Since $C^*(\CB/(\CB\setminus\CT),\CL,\theta,[\CI_\alpha])$ is a quotient of $C^*(\CB,\CL,\theta,\CI_\alpha;\CJ)$, we have that $C^*(\CB,\CL,\theta,\CI_\alpha;\CJ)$ has a quotient that contains a hereditary $C^*$-subalgebra that is isomorphic to $M_n(C(\T))$.
\end{proof}

A $C^*$-algebra $A$ has real rank zero if every self-adjoint element in the minimal unitization of $A$ can be approximated by invertible self-adjoint elements of the minimal unitization of $A$. 
The following is an easy consequence of Theorem \ref{equiv:Condition (K)}.

\begin{corollary}Let $(\CB,\CL,\theta, \CI_\af;\CJ)$ be a relative generalized Boolean dynamical system. If $C^*(\CB,\CL,\theta, \CI_\af;\CJ)$ is  purely infinite or has real rank zero, then $(\CB, \CL,\theta)$ satisfies Condition (K).
\end{corollary}

\begin{proof}
If  $C^*(\CB,\CL,\theta, \CI_\af; \CJ)$ is purely infinite, then 
$C^*(\CB,\CL,\theta,\CI_\alpha;\CJ)$ has no quotient containing a hereditary $C^*$-subalgebra that is isomorphic to $M_n(C(\T))$ for some $n\in\N$ since the property of being purely infinite passes to quotients and corners (see \cite[Propositions 4.3 and 4.17]{KR}). Thus, by Theorem~\ref{equiv:Condition (K)}, we have  $(\CB, \CL, \theta)$ satisfies Condition $(K)$.

If  $C^*(\CB,\CL,\theta, \CI_\af; \CJ)$ is of  real rank zero, then $C^*(\CB, \CL, \theta, \CI_\af;\CJ)$ has the ideal property by \cite[Theorem 2.6]{BP2}.  It then follows that $(\CB, \CL, \theta)$ satisfies Condition $(K)$  by  Theorem~\ref{equiv:Condition (K)}.
\end{proof}

\section{Minimality and simplicity}\label{Minimality and simplicity}
It follows from \cite[Theorem 7.4]{CaK2} that if the $C^*$-algebra of a relative generalized Boolean dynamical system $(\CB,\CL,\theta,\CI_\alpha;\CJ)$ is simple, then $\CJ=\CB_\reg$. We will in this section generalize \cite[Theorem 9.16]{COP} and characterize when the $C^*$-algebra of a generalized Boolean dynamical system $(\CB,\CL,\theta,\CI_\alpha)$ is simple (Corollary~\ref{simple}). But we begin with two leammas and a partly generalizing and strengthening \cite[Theorem 9.15]{COP}.

\subsection{Minimality}
If $\CI_1$ and $\CI_2$ are two ideals of a Boolean algebra $\CB$, then we denote by $\CI_1\oplus\CI_2$ the smallest ideal of $\CB$ that contains both $\CI_1$ and $\CI_2$. It is easy to see that
\[
	\CI_1\oplus\CI_2=\{A_1\cup A_2:A_1\in\CI_1,\ A_2\in\CI_2\}.
\]

\begin{lemma} \label{SH}
Let $(\CB,\CL,\theta)$ be a Boolean dynamical system and suppose $A\in\CB$. Then 
\[
	\CH(A):=\bigl\{B\in\CB:\text{there exists a finite subset }F\subseteq\CL^*\text{ such that }B\subseteq\bigcup_{\beta\in F}\theta_\beta(A)\bigr\}
\]
is the smallest hereditary ideal that contains $A$, and
\begin{align*}
	\CS(\CH(A))&:=\{B\in\CB:\text{there is an }n\in\N_0\text{ such that }\theta_\beta(B)\in\CH(A)\text{ for all }\beta\in\CL^n,\\
	&\qquad\text{and }\theta_\gamma(B)\in\CH(A)\oplus\CB_\reg\text{ for all }\gamma\in\CL^*\text{ with }|\gamma|<n\}
\end{align*}
is a saturated hereditary ideal that contains $A$.
\end{lemma}

\begin{proof}
It is straightforward to check that $\CH(A)$ is a hereditary ideal, and it is easy to see that if $\CH$ is a hereditary ideal and $A\in\CH$, then $\CH(A)\subseteq\CH$.

It is also straightforward to check that $\CS(\CH(A))$ is a saturated hereditary ideal. \end{proof}

For the proof of Lemma \ref{max tail}, the following notion of a partially defined topological graph will be useful.  

\begin{definition}(cf.\cite[Definition 4.6, 4.7]{Ka2006}) Let $E$ be a partially defined topological graph.
\begin{enumerate}
\item For $n \in \mathbb{N}\cup\{\infty\}$, a path $e \in E^n$ is called a {\it negative orbit} of $v \in E^0$ if $r(e)=v$ and $d(e) \in E^0_{sg}$ when $n < \infty$. 
\item For each negative orbit $e=(e_1, e_2, \cdots, e_n) \in E^n$ for $v \in E^0$, a {\it negative orbit space} $\operatorname{Orb}^-(v,e)$ is defined by 
$$\operatorname{Orb}^-(v,e)=\{v, d(e_1), d(e_2), \cdots, d(e_n)\} \subset E^0.$$
\end{enumerate} 
\end{definition}

\begin{lemma} \label{max tail}
Let $(\CB,\CL,\theta)$ be a Boolean dynamical system such that $\CB\ne\{\emptyset\}$. Then $(\CB,\CL,\theta)$ has a maximal tail.
\end{lemma}

\begin{proof}
Consider the partially defined topological graph $E:=E_{(\CB,\CL,\theta,\CR_\alpha)}$ constructed in Section~\ref{CKUT for GBDS}. Since $\CB\ne\{\emptyset\}$, we have that $E^0\ne\emptyset$. Choose $\chi \in E^0$. 
Let $e:=(e^{\af_n}_{\eta_n})_{n \geq 1}$ be a negative orbit of $\chi$.  We claim that 
\begin{align*}
	\CT&:=\{A\in\CB:\text{there exists }\beta\in\CL^*\text{ such that }\theta_\beta(A)\in\eta \text{ for some }\eta\in \operatorname{Orb}^-(\chi,e)\}
\end{align*}
is a maximal tail.  Clearly, we have $\emptyset \notin \CT$. We show that

(T2): Let $A \in \CB$ such that $\theta_\af(A) \in \CT$ for some $\af \in \CL$. Then, there is $\bt \in \CL^*$ such that $\theta_\bt( \theta_\af(A)) =\theta_{\af\bt}(A) \in \eta$ for some $\eta\in \operatorname{Orb}^-(\chi,e)$.  Thus, $A \in \CT$. 

(T3): Let $A\cup B  \in \CT$. Then   there is $\bt \in \CL^*$ such that $ \theta_\bt(A \cup B)=\theta_\bt(A) \cup \theta_\bt(B) \in \eta $ for some $\eta\in \operatorname{Orb}^-(\chi,e)$. Since $\eta$ is an ultrafilter,  either $\theta_\bt(A) \in \eta$ or $\theta_\bt(B) \in \eta$.  Hence, $A \in \CT$ or $B \in \CT$. 

(T4): Let $A \in \CT$ and $ B \in \CB$ with $A \subseteq B$. Then, there is $\bt \in \CL^*$ such that $\theta_\bt(A) \in \eta$ for some $\eta\in \operatorname{Orb}^-(\chi,e)$. Since 
$ \theta_\bt(A)  \subseteq \theta_\bt(B)$, $\theta_\bt(B) \in \eta$. Thus, $B \in \CT$. 

(T5): Let $A \in \CT$ be a regular set.  Then, there is $\bt \in \CL^*$ such that $\theta_\bt(A) \in \eta$ for some $\eta\in \operatorname{Orb}^-(\chi,e)$. If $\theta_\af(A) \notin \CT$ for all $\af \in \CL^* \setminus \{\emptyset\}$, then $\theta_\bt(\theta_\af(A)) \notin \eta$ for all $\eta \in \operatorname{Orb}^-(\chi,e)$ and all $\af, \bt \in \CL^*\setminus \{\emptyset\}$, a contradiction. Thus, $\theta_\af(A) \in \CT$ for some $\af \in \CL^* \setminus \{\emptyset\}$.

(T6):  Let $A, B \in \CT$. Then there exist $\bt, \bt' \in \CL^*$ such that $\theta_\bt(A) \in \eta$ and $\theta_{\bt'}(B) \in \eta'$ for some $\eta, \eta' \in \operatorname{Orb}^-(\chi,e)$. We may assume that $$\eta=r(e^{\af_i}_{\eta_i} \cdots e^{\af_j}_{\eta_j})(=f_{\emptyset[\af_{i,j}]}(g_{(\af_{i,j-1})\af_j}(\eta_j))) ~\text{and}~\eta'=d(e^{\af_i}_{\eta_i} \cdots e^{\af_j}_{\eta_j})(=h_{[\af_j]\emptyset}(\eta_j))$$ for some $1 \leq i,j \leq |e|$. Then, $\theta_{\bt\af_{i,j}}(A) \in \eta_{j} \cap \CI_{\af_{i,j}}$. Thus, $\theta_{\bt\af_{i,j}}(A) \cap \theta_{\bt'}(B) \in \eta'$, and hence, $\theta_{\bt\af_{i,j}}(A) \cap \theta_{\bt'}(B) \in \CT$.
\end{proof}

\begin{definition}
A Boolean dynamical system $(\CB,\CL,\theta)$ is \emph{minimal} if $\{\emptyset\}$ and $\CB$ are the only saturated hereditary ideals of $\CB$.
\end{definition}





\begin{proposition}\label{Minimality}
Let $(\CB,\CL,\theta,\CI_\alpha)$ be a generalized Boolean dynamical system. Then the following are equivalent. 
\begin{enumerate}
	\item $(\CB,\CL,\theta)$ is minimal.
	\item Either $\CB=\{\emptyset\}$ or $\CB\setminus\{\emptyset\}$ is the only maximal tail of $(\CB,\CL,\theta)$.
	\item If $A\in\CB\setminus\{\emptyset\}$, then $\CS(\CH(A))=\CB$.
	\item If $A,B\in\CB$, $x\in\CL^\infty$ and $A\ne\emptyset$, then there are a $C\in\CB_\reg$ such that $B\setminus C\in\CH(A)$, and an $n\in\N_0$ such that $\theta_{x_{1,n}}(B)\in\CH(A)$.
	\item If $A,B\in\CB$ and $A\ne\emptyset$, then there is a $C\in\CB_\reg$ such that $B\setminus C\in\CH(A)$ and such that there for every $x\in\CL^\infty$ is an $n\in\N_0$ such that $\theta_{x_{1,n}}(C)\in\CH(A)$.
	\item $\{0\}$ and $C^*(\CB,\CL,\theta,\CI_\alpha)$ are the only gauge-invariant ideals of $C^*(\CB,\CL,\theta,\CI_\alpha)$.
\end{enumerate}
\end{proposition}

\begin{proof}
The equivalence of (1) and (6) follows from \cite[Theorem 7.4]{CaK2}. We will show that $(1)\implies (2)\implies (3)\implies (4)\implies (5)$ and $\lnot (1)\implies \lnot (5)$.

$(1)\implies (2)$: Suppose $(\CB,\CL,\theta)$ is minimal and that $\CB\ne\{\emptyset\}$. According to Lemma~\ref{max tail}, $(\CB,\CL,\theta)$ then has a maximal tail. Suppose $\CT$ is a maximal tail. Then $\CB\setminus\CT$ is a saturated hereditary ideal of $\CB$. Since $(\CB,\CL,\theta)$ is minimal, it follows that $\CB\setminus\CT=\{\emptyset\}$, and thus $\CT=\CB\setminus\{\emptyset\}$.

$(2)\implies (3)$: Suppose (2) holds and that $A\in\CB\setminus\{\emptyset\}$. Then $\CS(\CH(A))$ is a saturated hereditary ideal of $\CB$. Suppose $\CS(\CH(A))\ne\CB$. Then, we see that $\CB / \CS(\CH(A)) \neq\{[\emptyset]\}$. It then follows from Lemma~\ref{max tail} that the quotient Boolean dynamical system $(\CB/\CS(\CH(A)),\CL,\theta)$ has a maximal tail $\CT$. Then
\[
	\widetilde{\CT}:=\{B\in\CB:[B]_{\CS(\CH(A))}\in\CT\}
\]
is a maximal tail of $(\CB,\CL,\theta)$ and therefore equal to $\CB\setminus\{\emptyset\}$. But that cannot be the case since $[A]_{\CS(\CH(A))}=[\emptyset]$. Hence, we must have that $\CS(\CH(A))=\CB$.

$(3)\implies (4)$: Suppose (3) holds, that $A,B\in\CB$, $x\in\CL^\infty$ and $A\ne\emptyset$. Then $B\in\CS(\CH(A))$. It follows from the description of $\CS(\CH(A))$ givne in Lemma~\ref{SH} that there is an $n\in\N_0$ such that $\theta_\beta(B)\in\CH(A)$ for all $\beta\in\CL^n$, and $\theta_\gamma(B)\in\CH(A)\oplus\CB_\reg$ for all $\gamma\in\CL^*$ with $|\gamma|<n$. If $n=0$ and we let $C=\emptyset$, then $C\in\CB_\reg$, $B\setminus C=B\in\CH(A)$ and $\theta_{x_{1,0}}(B)=\emptyset\in\CH(A)$. If $n>0$, then $\theta_{x_{1,n}}(B)\in\CH(A)$ and there is a $C\in\CB_\reg$ such that $B\setminus C\in\CH(A)$. Thus, (4) holds.

$(4)\implies (5)$: Suppose (4) holds, that $A,B\in\CB$, $x\in\CL^\infty$ and $A\ne\emptyset$. Then there are a $C\in\CB_\reg$ such that $B\setminus C\in\CH(A)$, and an $n\in\N_0$ such that $\theta_{x_{1,n}}(B)\in\CH(A)$. We then have that $B\cap C\in\CB_\reg$, $B\setminus (B\cap C)=B\setminus C\in\CH(A)$. Moreover $\theta_{x_{1,n}}(B\cap C)\subseteq \theta_{x_{1,n}}(B)\in\CH(A)$, which implies that $\theta_{x_{1,n}}(B\cap C)\in\CH(A)$. Thus, (5) holds.

$\lnot(1) \implies \lnot (5)$: Suppose that $\CI$ is a saturated hereditary ideal different from $\{\emptyset\}$ and $\CB$. Choose $A\in\CI\setminus\{\emptyset\}$ and $B\in\CB\setminus\CI$. Since $\CH(A)\subseteq\CI$, we have that if there is a $B'\in\CB$ such that $B'\setminus C\notin\CI$ for any $C\in\CB_\reg$, then (5) does not hold. Suppose that  for every $B'\in\CB$, there is a $C\in\CB_\reg$ such that $B'\setminus C\in\CI$. Suppose $C_1\in\CB_\reg$ and $B\setminus C_1\in\CI$. Since $B\notin\CI$, it follows that $C_1\notin\CI$. Since $C_1\in\CB_\reg$, we deduce that there is an $\alpha_1\in\CL$ such that $\theta_{\alpha_1}(C_1)\notin\CI$. We can then choose $C\in\CB_\reg$ such that $\theta_{\alpha_1}(C_1)\setminus C\in\CI$. Let $C_2:=C\cap\theta_{\alpha_1}(C_1) (\neq \emptyset)$. Since $\theta_{\alpha_1}(C_1)\notin\CI$, it follows that $C_2\notin\CI$. Since $C_2\in\CB_\reg$, we deduce that there is an $\alpha_2\in\CL$ such that $\theta_{\alpha_2}(C_2)\notin\CI$. Continuing like this, we can construct a sequence $(C_n,\alpha_n)_{n\in\N}$ such that we for each $n\in\N$ have $C_n\in\CB_\reg\setminus\CI$, $\alpha_n\in\CL$, $C_{n+1}\subseteq \theta_{\alpha_n}(C_n)$ and $\theta_{\alpha_n}(C_n)\setminus C_{n+1}\in\CI$. Let $x=\alpha_1\alpha_2\cdots$ and suppose $n\in\N$. Then $C_{n+1}\subseteq\theta_{x_{1,n}}(C_1)$. Since $C_{n+1}\notin\CI$, and therefore $C_{n+1}\notin \CH(A)$ it follows that $\theta_{x_{1,n}}(C_1)\notin\CH(A)$. We thus have that (5) does not hold.
\end{proof}

\subsection{Simplicity}
We now state our main result of Section \hyperref[Minimality and simplicity]{5}. It is a generalization of \cite[Theorem 9.16]{COP},  
\cite[Theorem 3.6]{CW2020} and \cite[Theorem 4.7]{Kang}. 

\begin{theorem} \label{simple}
Let $(\CB,\CL,\theta,\CI_\alpha)$ be a generalized Boolean dynamical system. Then the following are equivalent.
\begin{enumerate}
	\item Either $\CB=\{\emptyset\}$, or $\CB\setminus\{\emptyset\}$ is the only maximal tail of $(\CB,\CL,\theta)$ and $\CB\setminus\{\emptyset\}$ is not cyclic.
	\item $(\CB,\CL,\theta)$ is minimal and satisfies Condition (L).
	\item $(\CB,\CL,\theta)$ is minimal and satisfies Condition (K).
	\item $C^*(\CB,\CL,\theta,\CI_\alpha)$ is simple.
\end{enumerate}
\end{theorem}

\begin{proof}
The equivalence of (1) and (3) follows from Theorem~\ref{equiv:Condition (K)} and Proposition~\ref{Minimality}, the equivalence of (2) and (3) follows from Theorem~\ref{equiv:Condition (K)}. 

(2)$\implies$(4): Let $I$ be a nonzero ideal of $C^*(\CB,\CL,\theta, \CI_\af)$. Since $(\CB,\CL,\theta)$  satisfies Condition (L), $I$ contains $p_A$ for some $A \in \CB \setminus \{\emptyset\}$ by the Cuntz--Krieger uniqueness theorem \ref{thm:CKUT for GBDS}. Then, $\CH_I=\{A \in \CB: p_A \in I \}$ is a nonempty saturated hereditary ideal of $\CB$ by \cite[Lemma 7.2(1)]{CaK1}. Since $(\CB,\CL,\theta)$ is minimal, $\CH_I=\CB$. Thus, $I=C^*(\CB,\CL,\theta,\CI_\af)$.

(4)$\implies$(1): Suppose that $C^*(\CB,\CL,\theta,\CI_\alpha)$ is simple. Then, by Proposition \ref{Minimality}, either $\CB=\{\emptyset\}$ or $\CB\setminus\{\emptyset\}$ is the only maximal tail of $(\CB,\CL,\theta)$.
Suppose that  $\CT:=\CB\setminus\{\emptyset\}$ is a cyclic maximal tail.
Then, by Proposition \ref{cyclic maximal tails}, there is a $B\in\CT$ such that $p_{B}C^*(\CB,\CL,\theta, \CI_\af)p_{B}$ is isomorphic to $M_n(C(\mathbb{T}))$ for some $n\in\mathbb{N}$.  
This contradicts to the fact that $C^*(\CB,\CL,\theta,\CI_\alpha)$ is simple. Thus, $\CT=\CB\setminus\{\emptyset\}$ is not  cyclic.
\end{proof}

\end{document}